\documentclass[11pt]{article}

\usepackage[utf8]{inputenc}

\usepackage[usenames,dvipsnames]{xcolor}
\usepackage[margin=1.2in]{geometry}
\usepackage{graphicx}
\usepackage{amsmath, amsthm, pxfonts}
\usepackage{appendix}
\usepackage{caption}
\usepackage{subcaption, subfig}

\numberwithin{equation}{section}
\newtheorem{lem}{Lemma}[section]
\newtheorem{theo}{Theorem}[section]
\newtheorem{prop}{Proposition}[section]

\newtheorem{dfn}{Definition}[section]
\newtheorem{asmp}{Assumption}[section]
\newtheorem{cond}{Condition}[section]
\newtheorem{exam}{Example}[section]

\usepackage{amsmath, bbm, accents}
\usepackage{amsfonts}
\newcommand{\bbR}{\mathbb{R}}

\newcommand{\bbE}{\mathbb{E}}
\newcommand{\bbF}{\mathbb{F}}

\newcommand{\eps}{Z}

\newcommand{\voly}{\sigma^Y}

\newcommand{\calF}{\mathcal{F}}

\newcommand{\bbP}{\mathbb{P}}
\newcommand{\bbQ}{\mathbb{Q}}
\newcommand{\E}[2][]{\mathbb{E}_{#1}\left[#2\right]}

\newcommand{\crochet}[1]{\left\langle{#1}\right\rangle}
\newcommand{\accro}[1]{\left\{{#1}\right\}}
\newcommand{\paren}[1]{\left( {#1}\right)}

\newcommand{\abs}[1]{\left| {#1} \right|}

\newcommand{\ul}{\overline{l}}
\newcommand{\dl}{\underline{l}}

\title{Optimal discretization of hedging strategies\\
with directional views}
\author{Jiatu Cai$^{1}$, Masaaki Fukasawa$^{2}$, Mathieu Rosenbaum$^{3}$ and Peter Tankov$^{1,4}$
\\[0.4cm]
{\normalsize
$^1$ Laboratoire de Probabilit\'es et Mod\`eles Al\'eatoires,} \\ {\normalsize
    Universit\'e Paris Diderot (Paris 7)} \\
$~~$\\
{\normalsize
    $^2$ Department of Mathematics,} \\ {\normalsize
    Osaka University} \\ 
$~~$\\
{\normalsize $^3$ Laboratoire de Probabilit\'es et Mod\`eles Al\'eatoires,} \\ {\normalsize
    Universit\'e Pierre et Marie Curie (Paris 6)}\\    
$~~$\\
{\normalsize $^4$ International Laboratory of Quantitative Finance,} \\ {\normalsize
     National Research University Higher School of Economics, Moscow}
}

\date{\small{\textsf{\today}}}

\begin{document}
\maketitle
\begin{abstract}
\noindent We consider the hedging error of a derivative due to discrete trading in the presence of a drift in the dynamics of the underlying asset. We suppose that the trader wishes to find rebalancing times for the hedging portfolio which enable him to keep the discretization error small while taking advantage of market trends. 
Assuming that the portfolio is readjusted at high frequency,
we introduce an asymptotic framework in order to derive optimal discretization
strategies. More precisely, we formulate the optimization problem in terms of an
asymptotic expectation-error criterion.
In this setting, the optimal rebalancing times are given by the
hitting times of two barriers whose values can be obtained by solving a linear-quadratic optimal control problem. In specific
contexts such as in the Black-Scholes model, explicit expressions for the optimal rebalancing times can be derived.
\\


\end{abstract}
\section{Introduction}

In order to manage the risks inherent to the derivatives they buy and
sell, practitioners use continuous time stochastic models to compute
their prices and hedging portfolios. In the simplest cases, notably in
that of the so-called delta hedging strategy, the hedging portfolio
obtained from the model is a time varying self financed combination of cash
and the underlying of the option. We denote the price at time $t$ of the underlying
asset by $Y_t$ and assume it to be a one-dimensional
semi-martingale. Hence, in such situations, the outputs
of the model are the price of the option together with the number of underlying assets to hold in the hedging portfolio at any time $t$, denoted by $X_t$ (the weight in cash is then deduced from the self financing property). Therefore, assuming zero interest rates, the theoretical value of the model based hedging portfolio at the maturity of the option $T$ is given by
$$\int_0^T X_tdY_t.$$

\noindent Typically, the process $X_t$ derived from the model is a continuously varying semi-martingale, requiring continuous trading to be implemented in practice. This is of course physically impossible and would be anyway irrelevant because of the costs induced by microstructure effects. Hence practitioners do not use the strategy $X_t$, but rather a discretized version of it. This means the hedging portfolio is only rebalanced at some discrete times and thus is held constant between two rebalancing times. Let us denote by $(\tau^n_j)_{j\geq 0}$ an increasing sequence of rebalancing times over $[0,T]$ (the meaning of the parameter $n$ will be explained below). With respect to the target portfolio obtained from the model, the hedging error due to discrete trading $Z_T^n$ is therefore given by
\begin{equation*}
Z^n_T = \sum_{j=0}^{+\infty}X_{\tau^n_j}(Y_{\tau^n_{j+1}\wedge T} - Y_{\tau^n_{j}\wedge T})- \int_0^TX_tdY_t.
\end{equation*}
Thus, some important questions in practice are:
\begin{itemize}
\item What is the order of magnitude of $Z^n_T$ in the case of classical discretization strategies ? 
\item For a given criterion, how to optimize the rebalancing times ?
\end{itemize}

\noindent The most classical rebalancing scheme is that of equidistant trading dates of the form 
$$\tau_j^n=jT/n,~j=0,\ldots,n,$$
where $n$ represents the total number of trades on the period $[0,T]$. In this setting, the
first question has been addressed in details. There are two popular approaches to quantify the hedging error $Z^n_T$, both of them being asymptotic, assuming the rebalancing frequency $n/T$ tends to infinity (that is $n$ tends to infinity since $T$ is fixed). A first possibility is to use the $L^2$ norm, where one typically looks for asymptotic bounds of the form
\begin{equation*}
\bbE[(Z^n_T)^2]\leq cn^{-{\theta}},\quad n\to \infty.
\end{equation*}
Many authors have explored various aspects of this problem in this deterministic rebalancing dates framework. For European call and put options
in the Black-Scholes model, it is shown in \cite{bertsimas2000time} and \cite{zhang1999couverture} that the $L^2$ error has a convergence
rate $\theta = {1}$. For other options, the convergence rate
depends on the regularity of the payoff. For example, it is
shown in \cite{gobet2001discrete} that for binary options, the convergence rate is $\theta=1/2$. However, in
this context, the convergence rate $\theta = 1$ can be achieved by choosing a suitable non equidistant deterministic rebalancing grid defined by 
$$
\tau^n_j = T - T\paren{1 - j/n}^{1/\beta},
$$
with $\beta \in (0, 1]$ being the fractional smoothness in the Malliavin
sense of the option payoff, see \cite{geiss2002quantitative}. An
asymptotic lower bound for the $L^2$ error is given in
\cite{fukasawa2011asymptotically, fukasawa2012efficient} for a
general class of rebalancing schemes. \\ 

\noindent The second way to assess the hedging error is through
the weak convergence of the sequence of the suitably rescaled random variables $Z^n_T$. When $X$ and $Y$ are It$\hat{\text{o}}$ processes, the case of equidistant rebalancing dates has been investigated in this approach in \cite{bertsimas2000time,hayashi2005evaluating,rootzen1980limit}, where the following convergence in law is proved:
\begin{equation}
\sqrt{n}Z^n_T\xrightarrow{\mathcal{L}}{} \sqrt{\frac{{T}}{2}}\int_0^T\sigma^X_t\sigma^Y_tdB_t, \label{hm.eq}
\end{equation}
where $\sigma^X$ and $\sigma^Y$ are the volatilities of $X$ and $Y$ and $B$ is a Brownian motion independent of the other quantities.
The case where $X$ and $Y$ are processes with jumps is treated in \cite{tankov2009asymptotic}.\\

\noindent This asymptotic approach has also been recently used in the
context where the rebalancing times are random stopping times. Some
specific hitting times based schemes derived from a microstructure
model are investigated in \cite{robert2010microstructural}. In
\cite{fukasawa2011discretization}, the author works with quite general
sampling schemes based on stopping times. More precisely, for a given parameter $n$ driving the asymptotic, one considers an increasing sequence of stopping times
$$0=\tau_0^n\leq\tau_1^n\leq\ldots\leq\tau_j^n\leq\ldots$$ 
so that almost surely,
$\underset{j\rightarrow\infty}{\text{lim}}\tau_j^n=T$ (meaning in fact that the stopping times are all equal to $T$ for large enough $j$)
and $$ \sup_j (\tau^n_{j+1}-\tau^n_j)$$
tends to $0$ in a suitable sense as $n$ goes to infinity.
Under some regularity conditions on the (random) rebalancing times, a general limit theorem for the hedging error is obtained in \cite{fukasawa2011discretization}. It is shown that after suitable renormalization (specified in the next sections), the hedging error converges in law to a random variable of the form
\begin{equation}\label{eqn: stable limit fukasawa}
\frac{1}{3}\int_0^T s_tdY_t + \frac{1}{\sqrt{6}}\int_0^T \big(a_t^2 - \frac{2}{3}s_t^2\big)^{1/2}\sigma^Y_tdB_t.
\end{equation}
Here $B$ is a Brownian motion independent of all the other quantities
and the processes $s$ and $a$ can be interpreted as the
asymptotic local conditional skewness and kurtosis of the increments of the
process $X$ between two consecutive discretization dates (see next sections for details).\\

\noindent One can remark a crucial difference between the deterministic discretization schemes
associated to \eqref{hm.eq} and the random stopping times case leading to \eqref{eqn: stable limit fukasawa}. For deterministic
dates, the discretization error asymptotically behaves as a stochastic integral with respect to Brownian motion. Therefore, it is (essentially) centered.
In the case of random discretization dates, one may obtain a ``biased'' asymptotic hedging error because of the presence of the term
$$\int_0^T s_tdY_t.$$
Hence, if $s$ does not vanish and $Y$ has non zero drift, the asymptotic hedging error is no longer centered.\\

\noindent From a practitioner viewpoint, this is quite an interesting
property. Indeed, it shows that in the presence of market trends,
the trader may actually be compensated for the extra risk arising from
discrete trading, provided that the rebalancing dates are chosen in an
appropriate way.
Of course one may say this is not the
option trader's job to try to get a positive expected return with the
hedging strategy. However, knowing that there is anyhow a hedging
error, it seems reasonable to optimize it to the trader's benefit.\\

\noindent Hence we place ourselves in the asymptotic high frequency regime where $n$ is large and therefore
$$\underset{j}{\text{sup}}(\tau_{j+1}^n-\tau_j^n)$$
is small, meaning that the hedging error should be small. In this setting we address the second question raised above, that is finding the optimal times to rebalance the portfolio. To do so, we simply use an asymptotic expectation-error type criterion. More precisely, we wish to maximize the expectation of the hedging error under a constraint on its $L^2$ norm. This is quite in the spirit of \cite{sepp2013you}, where the author aims at finding an optimal hedging frequency to optimize the Sharpe ratio. Remark that in our context, the $L^2$ norm is more meaningful than the variance since the primary goal of the trader is to make the hedging error small. Our asymptotic approach goes as follows. First, we approximate the law of the renormalized hedging error by that in Equation \eqref{eqn: stable limit fukasawa}. Then we find the processes $a_t^*$ and $s_t^*$ which correspond to optimality in terms of our expectation-error criterion for the family of laws given by \eqref{eqn: stable limit fukasawa}. Finally, we show that we can indeed build a discretization rule which leads to the optimal $a_t^*$ and $s_t^*$ in the limiting distribution of the hedging error.\\

\noindent Note that using an asymptotic framework to design optimal discretizations of hedging strategies has been a quite popular approach in the recent years. Such method (although in a slightly different context) is in particular used in \cite{fukasawa2011asymptotically,fukasawa2011discretization,gobet2012almost} in the continuous setting whereas the case with jumps is investigated in \cite{rosenbaum2011asymptotically}. All these works aim at minimizing some form of transaction costs (typically the number of trades) under some constraint on the $L^2$ norm of the hedging error. Here we also put a constraint on the $L^2$ norm of the hedging error. However, instead of minimizing transaction costs, we maximize the expectation of the hedging error. Thus our viewpoint is that of a trader giving himself a lower bound on the quality of his hedge (the $L^2$ norm of the hedging error), but allowing himself to try to take advantage of market drifts provided the constraint is satisfied.\\

\noindent In practice, our work should probably only be considered as a benchmark. Indeed, we somehow make the highly unreasonable assumption that practitioners observe the drift. This is of course not realistic at all since any kind of statistical estimation of the drift is irrelevant in this high frequency setting. However, some practitioners still have views on the market and our work gives them a way to incorporate their beliefs in their hedging strategies.\\

\noindent The paper is organized as follows. In Section \ref{sec : framework},
we investigate the set of admissible discretization rules, that is those leading to a limiting law of the form \eqref{eqn: stable limit fukasawa}. In particular, we extend the examples provided in \cite{fukasawa2011discretization} by showing that the discretization rules based on hitting times of stochastic barriers are admissible. In Section \ref{sec: optimality}, we 
consider a first criterion for optimizing the trading times: the modified Sharpe ratio. It enables us to carry out very simple computations. However, the relevance of the modified Sharpe ratio being
in fact quite arguable, a more suitable approach in which we consider an expectation-error type criterion is investigated in Section \ref{sec: optimality2}. Using tools from linear-quadratic optimal control theory, explicit developments are provided in the Black-Scholes model in Section \ref{sec: bs}. Finally, the longest proofs are relegated to an appendix.

\section{Assumptions and admissible strategies}\label{sec : framework}
In this section we detail our assumptions on the processes $X$ and $Y$ together with the admissibility conditions
for the sampling schemes.
\subsection{Assumptions on the dynamics and admissibility conditions}
Let $(\Omega, \calF, \bbF, \bbP)$ be a filtered probability space. We
write $Y$ for the underlying asset. 
Let $T > 0$ stand for the maturity of the derivative to be hedged.
We assume that the benchmark
hedging strategy deduced from a theoretical model simply consists in
holding a certain number of units of the underlying asset, denoted by
$X$, and some cash in a self financed way, under zero interest rates.
Throughout the paper, we assume both $Y$ and $X$ are It$\hat{\text{o}}$ processes of the form
\begin{equation} \label{eq:itorep}
dY_t = b^Y_tdt + \sigma^Y_t dW^Y_t,~~dX_t = b^X_tdt + \sigma^X_t dW^X_t
\end{equation}
on $[0,T]$, 
where $W^X$ and $W^Y$  are  $\mathbb F$-Brownian motions which may be
arbitrarily correlated, and
 the coefficients of $X$ and $Y$ satisfy the following technical assumptions.
\begin{asmp}\label{assump_model} 
${}$
\begin{itemize}
\item The processes $b^Y$, $b^X$, $\sigma^Y$  and  $\sigma^X$ are
      adapted and  continuous  on $[0,T]$ almost surely.
\item The volatility process $\sigma^Y$ of $Y$  is positive  on $[0,T]$ almost
      surely.
\item The volatility process $\sigma^X$ of $X$  is positive  on $[0,T)$ almost
      surely.
\item The instantaneous Sharpe ratio $\rho = b^Y/\sigma^Y$ satisfies
\begin{equation*}
\mathbb{E}\big[\int_0^T\rho_t^2dt\big]<+\infty.
\end{equation*}
\end{itemize}
\end{asmp}
\begin{exam}[The Black-Scholes model]\label{example : bs} \upshape
The case that $b^Y_t = b Y_t$ and $\sigma^Y_t = \sigma Y_t$ with
 constants $b$ and  $\sigma > 0$ corresponds to the Black-Scholes
 model. The instantaneous Sharp ratio $\rho = b/\sigma$ is a
 constant. To hedge a call option with payoff $(Y_T-K)_+$ and strike $K > 0$, 
the standard theory suggests to use the so-called Delta hedging strategy:
\begin{equation*}
X_t = \Phi(d_1(t,Y_t)), \ \ d_1(t,y) = \frac{\log(y/K) + \sigma^2(T-t)/2}{\sigma\sqrt{T-t}},
\end{equation*}
where $\Phi$ stands for the distribution function of a standard Gaussian random variable.
By It$\hat{\text{o}}$'s formula, we see that $X$ is an
 It$\hat{\text{o}}$ process of the form \eqref{eq:itorep} with  
$W^X = W^Y$ and
\begin{equation*}
\begin{split}
& b^X_t = \phi(d_1(t,Y_t)) \left\{
\frac{\partial d_1}{\partial t}(t,Y_t) + \frac{\sigma^2}{2}\frac{\partial^2
d_1}{\partial y^2} (t,Y_t) Y_t^2 + b \frac{\partial d_1}{\partial
 y}(t,Y_t) Y_t\right\}+\frac{\sigma^2}{2}\big(\frac{\partial
d_1}{\partial y} (t,Y_t)\big)^2\phi'(d_1(t,Y_t))Y_t^2, \\
& \sigma^X_t = \sigma \phi(d_1(t,Y_t)) 
\frac{\partial d_1}{\partial y}(t,Y_t) Y_t,
 \end{split}
\end{equation*}
with $\phi$ the density of a standard Gaussian random variable. Almost surely, $Y_T\neq K$ and therefore
both $b^X$ and $\sigma^X$ are continuous on $[0,T]$ and $b^X_T =
 \sigma^X_T = 0$. Furthermore $\sigma^X$ is positive on
 $[0,T)$. Hence
Assumption \ref{assump_model} is satisfied.
 \end{exam}
\noindent As explained in the introduction, in practice, the trader cannot realize the theoretical strategy $X_t$ which typically implies continuous trading. Hence the quantity 
$$\int_0^T X_sdY_s$$ only represents a benchmark terminal wealth and $X_t$ is a benchmark hedging strategy.
Thus, we consider that this strategy is discretized over the stopping times 
\begin{equation*}
0=\tau^n_0\leq \tau^n_1\leq \cdots\leq \tau^n_j\leq\cdots,
\end{equation*}
so that for given $n$, almost surely, $\tau_j^n$ attains $T$ for $j$ large enough. Such array of stopping times is called a discretization rule. Consequently, if we define the discretized process $X^n$ by
\begin{equation*}
X^n_t  = X_{\tau^n_j}, \quad t\in [\tau^n_j, \tau^n_{j+1}), 
\end{equation*}
the hedging error $Z^n_T$ with respect to the benchmark strategy writes
\begin{equation*}
Z^n_T = \int_0^T(X^n_{s-} - X_s) dY_s. 
\end{equation*}

\noindent We now define the admissibility conditions for our discretization rules which we comment in the next subsection. 

\begin{cond}[Admissibility conditions]\label{Condition I} A discretization rule $(\tau_j^n)$ is admissible if
there exist continuous $\mathbb F$-adapted processes $a$ and $s$ satisfying
\begin{equation}\label{asmp: integrability a s}
\E{\int_0^T\big(1+(\rho_t)^2\big)(a_t^2+s_t^2)(\sigma^Y_t)^2dt}< \infty,
\end{equation}
and a positive sequence $\varepsilon_n$ tending to zero such that:
\begin{itemize}
\item  The first two moments of the renormalized hedging error $\varepsilon_n^{-1}Z^n_T$ converge
to those of a random variable of the form
\begin{equation}\label{eqn: stable limit}
Z^*_{a,s}=\frac{1}{3}\int_0^T s_tdY_t + \frac{1}{\sqrt{6}}\int_0^T\paren{a_t^2 - \frac{2}{3}s_t^2}^{1/2}\sigma^Y_tdB_t,
\end{equation}
that is,
\begin{equation}\label{moment2}
\bbE[\varepsilon_n^{-1}Z^n_T]\rightarrow \bbE[Z^*_{a,s}], \ \ 
\bbE[(\varepsilon_n^{-1}Z^n_T)^2]\rightarrow \bbE[(Z^*_{a,s})^2],
\end{equation}
where $B$ is a Brownian motion independent of all the other quantities.
\item Almost surely, the processes $a_t$ and $s_t$ satisfy $a_t^2 \geq
      s_t^2$, for all $t \in [0,T]$.
\end{itemize}
\end{cond}

\subsection{Comments on the admissibility conditions}
\noindent Equation \ref{asmp: integrability a s} is simply a technical integrability condition. 
We now give the interpretation of the sequence $\varepsilon_n$. Recall that for fixed $n$, we deal
with an increasing sequence of stopping times $(\tau_j^n)$ over $[0,T]$. Typically, $\varepsilon_n^2$ will represent
the order of magnitude of the interarrival time $\tau_{j+1}^n-\tau_j^n$. For example, in the case of equidistant trading times
with frequency $n/T$, $\varepsilon_n$ can simply be taken equal to $n^{-1/2}$. In the case of the hitting times based scheme consisting in rebalancing the portfolio
each time the process $X$ has varied by $\nu_n$, where $\nu_n$ is a deterministic sequence tending to zero, one can choose $\varepsilon_n=\nu_n$ (since the order of magnitude of the time interval between two hitting times is $\nu_n^2$). 
\\

\noindent The specific form \eqref{eqn: stable limit} may
appear rather ad hoc at first sight. However, it is in fact quite
natural. Indeed, Proposition~\ref{proposition: stable conv} below, which is proved in Appendix and used to show the main result of the next subsection, indicates that as soon as the quadratic covariations $\varepsilon_n^{-2}\crochet{Z^n}$
and $\varepsilon_n^{-1}\crochet{Z^n, Y}$ have regular limits, the form \eqref{eqn: stable limit} appears for the weak limit of the  
renormalized hedging error. So the idea for this admissibility condition is that in our asymptotic approach, we want to work in regular cases where the renormalized hedging error can be approximated by a random variable of the form \eqref{eqn: stable limit}. However, our asymptotic optimality criterion will be based on the first two moments of the renormalized hedging error only. Therefore, we just require these first two moments to be asymptotically close to those of a random variable of the form \eqref{eqn: stable limit} (in particular we do not impose the convergence in law of the renormalized hedging error towards $Z^*_{a,s}$, although this is the underlying idea behind this admissibility condition). We now give Proposition \ref{proposition: stable conv}.

\begin{prop}\label{proposition: stable conv}
If there exist a sequence $\varepsilon_n \to 0$ and continuous processes $s$ and
 $a$ such that
\begin{align}
&\varepsilon_n^{-2}\crochet{Z^n}\rightarrow\frac{1}{6}\int_0^{\cdot}a_u^2(\sigma^Y_u)^2du,\label{asmp: limit 2} \\
&\varepsilon_n^{-1}\crochet{Z^n, Y}\rightarrow \frac{1}{3}\int_0^{\cdot}s_u(\sigma^Y_u)^2du,\label{asmp: limit 3}
\end{align} 
uniformly in probability on  $[0,T]$, then $\varepsilon_n^{-1}Z^n$ converges in law to 
\begin{equation}\label{condstabconv}
\frac{1}{3}\int_0^\cdot s_tdY_t + \frac{1}{\sqrt{6}}\int_0^\cdot\paren{a_t^2 - \frac{2}{3}s_t^2}^{1/2}\sigma^Y_tdB_t
\end{equation}
in $C[0,T]$. In particular the convergence in law of $\varepsilon_n^{-1}Z^n_T$  to $Z_{a,s}^\ast$  defined by \eqref{eqn: stable limit}
 holds.
If in addition, 
\begin{equation}
\label{Condition add}
\varepsilon_n^{-4/3}\sup_{j\geq 0}(\tau^n_{j+1}\wedge T_0  - \tau^n_j \wedge T_0)\to 0
\end{equation}
in probability, for all $T_0 \in [0,T)$, then almost surely $a_t^2 \geq s^2_t$ for
 all  $t \in [0,T]$. 
\end{prop}

\noindent We now consider the processes $a^2_t$ and $s_t$
appearing in the admissibility conditions. We place ourselves in the situation where Proposition \ref{proposition: stable conv} can be applied. In that case, an inspection of the proof of this lemma shows that the inequality $a_t^2 \geq s^2_t$ essentially follows from the elementary fact that
$\bbE[\Delta^2]\bbE[\Delta^4] \geq \bbE[\Delta^3]^2$ for a general random variable $\Delta$. Indeed, $a^2_t$ and $s_t$
are respectively related to the local third and fourth conditional moments of the increments of $X$. Proposition~\ref{prop: intuition a s} below, which is proved in Appendix and used to show the main result in the next subsection, somehow illustrates the connections between $a^2_t$ and $s_t$ and the conditional moments. Thus we give it now.
Let $\Delta_{j,n} = X_{\tau^n_{j+1}} - X_{\tau^n_j}$ be the increment of
$X$ between $\tau^n_j$ and $\tau^n_{j+1}$ and  $N^n_t$  be the number of rebalancing times until time $t$:
\begin{equation*}
N^n_t = \max\accro{j \geq 0 \vert \tau^n_j \leq t }.
\end{equation*}
The following proposition holds.

\begin{prop}\label{prop: intuition a s} 
Let $\varepsilon_n$ be a positive sequence tending to $0$ and 
$s$ and $a$ be continuous processes. Assume the following:
\begin{itemize}
\item The family of random variables 
\begin{equation} \label{sup diff}
\varepsilon_n^{-4} \sup_{t \in [0,T]} |X^n_t - X_t|^4
\end{equation}
is uniformly integrable.

\item The following uniform convergences in probability on $[0,T_0]$ hold for all $T_0 \in [0,T)$:
\begin{equation} \label{asmp: suff}
\begin{split}
\varepsilon_n^{-1}&\sum_{j=0}^{N^n_\cdot}\kappa_{\tau^n_j}\mathbb{E}\big[
\Delta_{j, n}^3
 \big| \mathcal{F}_{\tau^n_j}
\big]
\to -\int_0^\cdot s_u (\sigma^Y_u)^2{d}u,\\
\varepsilon_n^{-2}&\sum_{j=0}^{N^n_\cdot}\kappa_{\tau^n_j}
\mathbb{E}\big[
\Delta_{j, n}^4
 \big| \mathcal{F}_{\tau^n_j}
\big]
\to \int_0^\cdot a_u^2 (\sigma^Y_u)^2{d}u,
\end{split}
\end{equation}
where $\kappa_u=(\sigma^Y_u/\sigma^X_u)^2$. 
\end{itemize}
Then the convergences \eqref{asmp: limit 2}, 
\eqref{asmp: limit 3} and \eqref{Condition add} hold.
\end{prop}

\noindent 
Proposition~\ref{prop: intuition a s} is useful to obtain the convergences \eqref{asmp: limit 2} and \eqref{asmp: limit 3} for a given discretization rule since it is usually easy to have approximate values of the conditional moments of the increments.
We actually apply this approach in the proof of the main result of the next subsection.

\subsection{Examples of admissible discretization rules}

We show in this section that the most common discretization rules are admissible. We start with hitting times based schemes. We have the following result.

\begin{prop}[Hitting times based discretization rule]\label{prop : hitting strats}
Let $\varepsilon_n$ be a positive sequence tending to zero and
$\dl$ and $\ul$ be two adapted processes which are positive and 
continuous on $[0,T]$ almost surely with
\begin{equation}\label{regbarrier} \E{\int_0^T\big(1+(\rho_t)^2\big)(\ul_t\vee\dl_t)^2(\sigma^Y_t)^2dt} < \infty.
\end{equation}
The discretization rule based on the hitting times of $\varepsilon_n \dl_{t}$ or $\varepsilon_n \ul_{t}$ by the process $X$:
\begin{equation}\label{def: hitting strat}
\tau^n_{j+1} = \inf\accro{t>\tau^n_j: X_t \notin (X_{\tau^n_j}- \varepsilon_n \dl_{t},X_{\tau^n_j}+\varepsilon_n \ul_{t} )} \wedge T
\end{equation}
is admissible. Moreover, we can take
\begin{equation}\label{prop: hitting strat}
s_t = \dl_t - \ul_t,\quad a_t^2 =(s_t)^2+\dl_t\ul_t
\end{equation}
and we also have the convergence \eqref{condstabconv} and therefore the convergence in law of $Z_T^n$ towards $Z^\ast_{a,s}$ defined by \eqref{eqn: stable limit}.
\end{prop}
\noindent
It is interesting to note here that the limit $Z^\ast_{a,s}$ does not depend on the structure of $X$.\\

\noindent This result is particularly important since many traders monitor the values of the increments of their so-called delta (which corresponds to the process $X$) in order to decide when to rebalance their portfolio. Thus they are indeed using hitting times based strategies. This proposition notably extends the weak convergence results in \cite{fukasawa2011discretization} since it shows that not only constant barriers (between  $\tau^n_j$ and $\tau^n_{j+1}$) but also time varying stochastic barriers can be considered. This will be very useful in the next sections since our optimal discretization rules will correspond to hitting times of such barriers. Furthermore, in the proofs, the assumption that the time varying barriers satisfy \eqref{regbarrier} will enable us to deduce quite easily some relevant integrability properties for the hedging error (which would be harder to obtain with locally constant barriers).\\ 

\noindent Now remark that under the condition $a_t^2 > s_t^2$, we can always find some positive processes $\ul_t$ and $\dl_t$ such that \eqref{prop: hitting strat}
is satisfied. Indeed, it is easy to see that the real numbers $\ul_t$ and $-\dl_t$ can be taken as
the roots of the quadratic equation $x^2 + s_tx +s_t^2 - a_t^2 = 0$. Under
the condition $a_t^2 > s_t^2$, this equation admits two nonzero roots
with different signs. Therefore, another interesting property of hitting times based schemes is the following.

\begin{lem}\label{lem : hitting strat univ}
For any pair of limiting processes $s$ and $a$ satisfying \eqref{asmp: integrability a s} and $a_t^2>s_t^2$, we can always build a corresponding admissible discretization rule based on hitting times as in \eqref{def: hitting strat}-\eqref{prop: hitting strat}.
\end{lem}
\noindent Consequently, if one has some processes $a_t$ and $s_t$ as targets, Lemma \ref{lem : hitting strat univ} implies that a strategy giving rise
to these processes in the limiting distribution \eqref{eqn: stable limit} can be found. We will work in this framework in Section \ref{sec: optimality2}.
Remark that there are infinitely many strategies for which the hedging error converge in law to some $Z^*_{a,s}$ with the same $a$ and $s$ as limiting processes. The hitting time strategy is an efficient one among them, in the sense that it requires the least number of rebalancing in an asymptotic sense, see \cite{fukasawa2011discretization} for the detail.\\

\noindent Another classical discretization rule is given by equidistant trading times.
Here, the integrability property \eqref{moment2} in the admissibility conditions does not hold in full generality. Compared to the hitting times setting, this is because the deviations of
the benchmark strategy are not explicitly controlled by the barriers. Nevertheless, the following example describes a reasonable framework under which such discretization rule is admissible. 

\begin{prop}[Equidistant sampling discretization rule]\label{eqstrat}
Consider the hedging strategy of a European option with payoff $h(Y_T)$ and replace Assumption \ref{assump_model} by that the underlying $Y_t$ follows a diffusion process of the form
\begin{equation*}
dY_t = b(t, Y_t)Y_tdt+\sigma(t,Y_t)Y_t dW_t,
\end{equation*}
with $b$, $\sigma$ and $h$ some deterministic functions satisfying the regularity assumptions p.21-23 in \cite{zhang1999couverture} (allowing in particular for call and put in the Black-Scholes model). Define the delta hedging portfolio:
$$
X_t = \frac{\partial P}{\partial y} (t,Y_t),\text{ with }P(t,y) = \mathbb
E_{(t,y)}^{\mathbb{Q}}[h(Y_T)],
$$ where $\bbE^{\mathbb{Q}}$ denotes the expectation operator under
the risk neutral probability. Let $\varepsilon_n$ be a positive sequence tending to zero.
Then the equidistant trading times discretization rule:
\begin{equation*}
\tau^n_j =j\varepsilon_n^2, \quad j=0, \ldots, n, \ldots
\end{equation*}
is admissible (under the original measure). Moreover, we can take
\begin{equation*}
s_t=0, \quad a_t^2 = 3(\sigma^X_t)^2.
\end{equation*}
\end{prop}

\noindent The proof of Proposition \ref{eqstrat} follows easily from previous works. We can first obtain the convergence in law towards $Z^*_{a,s}$ using for example the results of \cite{hayashi2005evaluating}. Indeed, up to localization, we can assume that $\sigma^X$, $\sigma^Y$, $b^X$ and $b^Y$ are bounded. Then the integrability conditions in the mentioned reference are obviously satisfied and the convergence follows. 
For \eqref{moment2}, it suffices to use Theorem  2.4.1 in
\cite{zhang1999couverture} where the convergence of the $L^2$ norm of the
 normalized error under the original measure is provided.\\

\noindent Finally, note that the discretization rule based on equidistant trading times will not be of interest for us since the associated $s_t$ process vanishes  and so the expectation of the limiting variable is zero.

\section{Asymptotic optimality: a preliminary approach}\label{sec: optimality}

Our viewpoint is that the trader's priority is to get a small hedging error. However, once this error is suitably controlled, he may try to take
advantage of the directional views he has on the market. Hence,
adopting the asymptotic approximation under which the first two moments of the renormalized hedging error are given by
those of $Z^{*}_{a,s}$, we aim at maximizing $\bbE[Z^{*}_{a,s}]$ while keeping $\bbE[(Z^{*}_{a,s})^2]$ reasonably small. This very problem is treated in Section \ref{sec: optimality2}.\\

\noindent Here, as a first step, we consider the approximation for $\bbE[(Z^{*}_{a,s})^2]$ given by $\bbE[(Z^{*,c}_{a,s})^2]$, where $Z^{*,c}_{a,s}$ denotes the sum of the two integrals with respect to the Brownian motions $W^Y$ and $B$ in the definition of $Z^{*}_{a,s}$ in Equation \eqref{eqn: stable limit}, that is
\begin{equation*}
Z^{*,c}_{a,s}=\frac{1}{3}\int_0^T s_t\sigma_t^YdW^Y_t + \frac{1}{\sqrt{6}}\int_0^T\paren{a_t^2 - \frac{2}{3}s_t^2}^{1/2}\sigma^Y_tdB_t.
\end{equation*}
To do so, we place ourselves in this section under the additional admissibility condition that the renormalized hedging error weakly converges in the sense of \eqref{condstabconv} and we take $s_t$ and $a_t^2$ as the processes in the limit \eqref{condstabconv} (so $s_t$ and $a_t^2$ are uniquely defined). Replacing $\bbE[(Z^{*}_{a,s})^2]$ by $\bbE[(Z^{*,c}_{a,s})^2]$ is technically very convenient but in practice quite arguable since this approximation is meaningful only when the drift is small. However, our aim here is only to have a first rough idea about the form of the optimal discretization rules. Since we wish to get
the moment of order one large while that of order two remains controlled, we consider that we want to maximize the so-called modified Sharpe ratio $S$ defined by 
\begin{equation*}
S=S(a,s)=\frac{\bbE[Z^{*}_{a,s}]}{\sqrt{\bbE[(Z^{*,c}_{a,s})^2]}}. 
\end{equation*}
This ratio is said to be modified since we use $\bbE[(Z^{*,c}_{a,s})^2]$ instead of the variance of $Z^*_{a,s}$.\\ 

\noindent Hence we are looking for strategies which maximize $S$. To do so, 
we now introduce the notion of nearly efficient (modified) Sharpe ratio.
\begin{dfn}[Nearly efficient Sharpe ratio]
The value $S^*\in\mathbb{R}$ is said to be a nearly efficient Sharpe ratio if:
\begin{itemize}
\item For any admissible discretization rule with associated limiting processes $a$ and $s$, the associated modified Sharpe ratio $S(a,s)$ satisfies
$$S(a,s)\leq S^*.$$
\item For any $\eta>0$, there exists a discretization rule with associated limiting processes $a$ and $s$ such that
$$S(a,s)\geq S^*-\eta.$$
\end{itemize}
\end{dfn} 
\noindent We only consider nearly efficient ratios since our strategies will not enable us to attain exact efficiency (which would corresponds to $\eta=0$ in the previous definition). Of course the slight difference between efficient and nearly efficient ratios has no importance in practice.\\

\noindent In our setting, for any limiting variable $Z^{*}_{a,s}$, we have
$$S(a,s)=\frac{\bbE\Big[\frac{1}{3}\int_0^T s_tb^Y_tdt\Big]}{\Big(\bbE\Big[\frac{1}{9}\int_0^Ts_t^2(\sigma^Y_t)^2dt+\frac{1}{6}\int_0^T\big(a_t^2-\frac{2}{3}s_t^2\big)(\sigma^Y_t)^2dt\Big]\Big)^{1/2}}.$$
Now, the admissibility condition $a_t^2 \geq s_t^2$ implies
$$S(a,s)\leq\frac{\sqrt{6}}{3}\frac{\bbE\Big[\int_0^T s_tb^Y_tdt\Big]}{\Big(\bbE\Big[\int_0^Ts_t^2(\sigma^Y_t)^2dt\Big]\Big)^{1/2}}$$
and Cauchy-Schwarz inequality gives  
$$S(a,s)\leq \frac{\sqrt{6}}{3}\Big(\bbE\Big[\int_0^T\big(\frac{b^Y_t}{\sigma^Y_t}\big)^2dt\Big]\Big)^{1/2}.$$
This provides an upper bound for the modified Sharpe ratio. We now wish to find a discretization rule enabling to (almost) attain this upper bound.
To achieve this, our rule must be so that for the associated processes $a_t$ and $s_t$, the inequalities used above ($a_t^2 \geq s_t^2$ and Cauchy-Schwarz) become almost equalities. This means that $a_t$ should be close to $s_t$ and $s_t$ essentially proportional to $b_t^Y/(\sigma^Y_t)^2$. Furthermore, we want the product $s_tb^Y_t$ to be essentially positive in order to get a positive modified Sharpe ratio. If we look for this rule among the hitting times based schemes specified by two processes $(\dl_t,\ul_t)$, Lemma \ref{lem : hitting strat univ} implies that
\begin{itemize}
\item the difference $\dl_t-\ul_t$ should be essentially proportional to $b_t/(\sigma^Y_t)^2$,
\item the product $\dl_t\ul_t$ should be negligible compared to $(\dl_t-\ul_t)^2$,
\item the term $(\dl_t-\ul_t)b^Y_t$ should be essentially positive.
\end{itemize}  
From these remarks together with Proposition \ref{prop : hitting strats}, we easily deduce the following theorem.

\begin{theo}\label{theo_unbiased}
Suppose that for all $t\leq T$, $b^Y_t\neq 0$. Then the value
\begin{equation*}
\frac{\sqrt{6}}{3}\Big(\bbE\Big[\int_0^T\big(\frac{b^Y_t}{\sigma^Y_t}\big)^2dt\Big]\Big)^{1/2}
\end{equation*}
is a nearly efficient Sharpe ratio. It is approximately
attained by the discretization rule defined for $\lambda>0$ by
\begin{equation}\label{discrule1}
\tau^{n,\lambda}_{j+1}= \inf \Big\{t > \tau^{n,\lambda}_j; X_t - X_{\tau^{n,\lambda}_j} =  -\frac{b^Y_{t}}{(\voly_{t})^2}e^{\lambda}\varepsilon_n\text{ or }\frac{b^Y_{t}}{(\voly_{t})^2}e^{-\lambda}\varepsilon_n\Big\}, \quad \tau^n_0 = 0.
\end{equation}
Indeed, 
\begin{equation*}
\lim_{\lambda\to +\infty} S(\lambda) = \frac{\sqrt{6}}{3}\Big(\bbE\Big[\int_0^T\big(\frac{b^Y_t}{\sigma^Y_t}\big)^2dt\Big]\Big)^{1/2},
\end{equation*}
where $S(\lambda)$ denotes the modified Sharpe ratio obtained for the law of the variable $Z^*_{a,s}$ associated to the discretization
rule \eqref{discrule1} with parameter $\lambda$.
\end{theo}

\noindent This result provides simple and explicit strategies for optimizing the modified Sharpe ratio. It is also very easy to interpret. Indeed, we see that in order to take advantage of the drift, one needs to consider asymmetric barriers. The limitation is that we do not really control accurately the magnitude of the hedging error at maturity.\\

\noindent The asymptotic setting simply means that we require
$\lambda$ to be quite large while $e^{\lambda}\varepsilon_n$ is
small. When using such discretization rule in practice, it is reasonable to consider
that the trader fixes a maximal value for the asymmetry
between the barriers controlled by $\lambda$. This way he can choose
the parameter $\lambda$. Then $\varepsilon_n$ is set to match the bound on $\bbE[(Z^{*,c}_{a,s})^2]$ that the trader does not want to exceed. 



\section{Asymptotic expectation-error optimization}\label{sec: optimality2}

In this section, we now consider a natural expectation-error type criterion in order to optimize our discretization rules. To do so, we work in an asymptotic setting where we are looking for discretization rules which are optimal in the expectation-error sense for their associated limiting random variable
$Z^*_{a,s}$. Before giving our main result, we introduce some definitions inspired by classical portfolio theory. 
\begin{dfn}[Non dominated couple]\label{def2}
A couple $(m,v)\in(\mathbb{R}^+)^2$ is said to be non dominated if there exists no admissible discretization rule such that its associated limiting random variable $Z^*_{a,s}$ satisfies
$$\bbE[Z^*_{a,s}]\geq m,~~\bbE[(Z^*_{a,s})^2]<v.$$ The set of non dominated couples is called the non domination domain.
\end{dfn}
\begin{dfn}[Nearly efficient couple]\label{def3}
A couple $(m,v)\in(\mathbb{R}^+)^2$ is said to be nearly efficient if it is non dominated and for any $\eta>0$, there exists an admissible discretization rule such that its associated limiting random variable $Z^*_{a,s}$ satisfies 
$$\bbE[Z^*_{a,s}]=m,~~\bbE[(Z^*_{a,s})^2]\leq v+\eta.$$
It is efficient if we can take $\eta=0$.
\end{dfn}
\noindent We introduce the set $\mathcal{Z}_T$ of random variables of the form
\begin{equation}\label{solprob}
Z_{T,s}=\frac{1}{3}\int_0^T s_tdY_t + \frac{1}{3\sqrt{2}}\int_0^T s_t\sigma^Y_tdB_t, 
\end{equation}
where $B$ is a Brownian motion independent of $\calF$ and $s_t$ is an adapted continuous process such that 
\begin{equation}\label{regul}
\bbE\big[\int_0^T\big(1+(\rho_t)^2\big)s_t^2(\sigma^Y_t)^2dt\big]< \infty.
\end{equation} 
We also define the notions of non dominated and efficient couples with respect to $\mathcal{Z}_T$. The definitions are the same as Definition \ref{def2} and Definition \ref{def3} except that we replace
``admissible discretization rule'' by ``process $s$ satisfying \eqref{regul}'' and ``its associated limiting random variable $Z^*_{a,s}$" by $Z_{T,s}$.\\

\noindent We can now state our main result which enables us to compute efficient discretization rules. 

\begin{theo}\label{theo : main result}
The following results hold:
\begin{itemize}
\item The non domination domain coincides with the non domination domain with respect to $\mathcal{Z}_T$.
\item Let $(m^*,v^*)$ be an efficient couple with respect to $\mathcal{Z}_T$, with associated optimal process $s^*$. Then $(m^*,v^*)$ is a nearly efficient couple. More precisely, let $\delta>0$ and $(\dl^\delta_t, \ul^\delta_t)$ be defined by
\begin{equation*}
\dl^\delta_t - \ul^\delta_t = s^*_t, \quad (\dl^\delta_t)^2 -\dl^\delta_t\ul^\delta_t + (\ul^\delta_t)^2 = (s^*_t)^2 + \frac{6\delta}{(\sigma^Y_t)^2},
\end{equation*}
that is
\begin{equation}\label{eqn : opt barriers}
\dl^\delta_t =  \sqrt{\frac{(s^*_t)^2}{4} + \frac{6\delta}{(\sigma_t^Y)^2}} + \frac{s^*_t}{2},  \quad \ul^\delta_t = \sqrt{\frac{(s^*_t)^2}{4} + \frac{6\delta}{(\sigma_t^Y)^2}} - \frac{s^*_t}{2}.
\end{equation}
Then the hitting times based discretization rule specified through the barriers $(\dl^\delta_t, \ul^\delta_t)$ 
satisfies
\begin{equation*}
\bbE[Z^*_{a,s}] = m^*,\quad  \bbE[(Z^*_{a,s})^2] = v^* +\delta T.
\end{equation*}
\end{itemize} 
\end{theo}

\noindent We have therefore reduced the impulse control problem of finding the optimal rebalancing times to a classical expectation-error optimization with continuous dynamics. The solutions of this problem can be obtained by solving for $\mu>0$
\begin{equation*}
\inf_{(s_t)}\big\{-\bbE[Z_{T,s}]+\mu\bbE[(Z_{T,s})^2]\big\},
\end{equation*}
for which we can apply the theory of linear-quadratic optimal control, see for example \cite{lim2002mean,zhou2000continuous}. As shown in the next section, we can even obtain closed formulas in the case where the underlying has deterministic drift and volatility. Note that again, our barriers strategies enable us to attain only nearly efficient couples. Indeed, reaching efficient couples would lead to the use of degenerate barriers with $\delta=0$. This does not make sense in practice, however $\delta$ can of course be selected small.\\

\noindent In practice, once he has chosen the target nearly efficient couple he wants to reach, the trader needs to select $\delta$ and $\varepsilon_n$. Two ideas enabling to avoid microstructure effects seem natural and easy to implement:
\begin{itemize}
\item Fix a minimal time between two rebalancings $t_{min}$. After a rebalancing at a random time say $\tau$, wait $t_{min}$ and then apply the strategy with $\delta=0$ (that is rebalance immediately if at $t=\tau+t_{min}$, $X_t-X_{\tau}$ is not inside the interval $(-\varepsilon_n \dl^0_t ,\varepsilon_n \ul^{0}_t)$ and wait for the exit time otherwise). The parameter $\varepsilon_n$ can be chosen according to the average number of transactions the trader is willing to make.
\item Fix (roughly) a minimal distance for the closest barrier after a rebalancing. Then compute $\delta$ and $\varepsilon_n$ according to the general level of volatility $\sigma_t^Y$ so that they (approximately) lead to this bound and the average number of transactions the trader is willing to make.
\end{itemize}

\section{One explicit example : Black-Scholes model with time varying coefficients}\label{sec: bs}

In this section, we explain how our method can be applied in practice through the simple example of the Black-Scholes model with time varying coefficients. So we assume the underlying follows the dynamics
$$ dY_t = Y_t(b_t dt + \sigma_t dW_t),
$$
where $b_t$ and $\sigma_t$ are continuous deterministic functions. We also assume $b_t$ and $\sigma_t$ do not vanish. Using the theory of linear-quadratic optimal control, we give an explicit solution for the problem of designing optimal rebalancing times in this specific setting.

\subsection{Explicit formulas}

\noindent We aim at finding the efficient couples for the controlled random variables of the form 
$Z_{T,s}$ as in \eqref{solprob}. 
Following \cite{zhou2000continuous}, such problem is classically recast as follows: solving for any $\mu>0$ the optimization problem
\begin{equation*}
\inf_{(s_t, 0\leq t\leq T)} -\bbE[{Z_{T,s}}] + \mu \bbE[(Z_{T,s})^2]=\inf_{(s_t, 0\leq t\leq T)} \mu \bbE\big[\big(Z_{T,s}-\frac{1}{2\mu}\big)^2\big]-\frac{1}{4\mu}.
\end{equation*}
Let us define the family of processes of the form
$$  
d\check Z_t = s_tY_t(\tilde b_tdt + \tilde\sigma_t dW_t),~~\check Z_0 =0,
$$  
with $\tilde b_t=b_t/3$, $\tilde\sigma_t=\sigma_t/3$ and $s_t$ adapted continuous.
Using obvious computations, the independence between the process $B$ in Equation \eqref{solprob} and $\mathcal{F}$, and the fact that $s_t$ is $\mathcal{F}$-adapted, we get $\bbE[\check Z_T]=\bbE[Z_{T,s}]$ and
$$\mu\bbE\big[\big(\check Z_T-\frac{1}{2\mu}\big)^2\big]=\mu\bbE\big[\big(Z_{T,s}-\frac{1}{2\mu}\big)^2\big]-\frac{\mu}{18}\bbE\big[\int_0^T(s_t\sigma_t Y_t)^2dt\big].$$
Hence, we can equivalently solve
\begin{equation*}
\inf_{(s_t, 0\leq t\leq T)}\bbE\big[ \mu \tilde Z_T^2+\frac{\mu}{2}\int_0^T(s_t\tilde\sigma_t Y_t)^2dt\big],
\end{equation*}
with 
$$  
d\tilde Z_t = s_tY_t(\tilde b_tdt + \tilde\sigma_t dW_t),~~\tilde Z_0 =-\frac{1}{2\mu}.
$$  

\noindent Using the results of \cite{zhou2000continuous} which are summarized in Theorem \ref{theoxyz} in Appendix \ref{append: LQ}, the optimal control $s_t^*$ and optimally controlled process $\tilde Z^*_t$ satisfy
\begin{equation*}\label{s_star}
s^*_tY_t = -\frac{1}{\tilde b_t}\frac{\dot{P_t}}{P_t}\tilde Z^*_t,
\end{equation*}
where $P_t$ is the solution of the (ordinary) differential equation
$$\dot{P_t}=\rho_t^2\frac{P_t^2}{P_t+\mu},~~P_T=2\mu,$$
with $\rho_t = b_t/\sigma_t$. The solution of this equation is given by
$$P_t=\frac{\mu}{L\Big(\frac{1}{2}\text{exp}\big(\int_t^T\rho_s^2ds+\frac{1}{2}\big)\Big)},$$
with $L$ is the inverse function of $x\mapsto xe^x$. Moreover, the optimal process $\tilde Z^*$ satisfies 
$$\frac{d\tilde Z^*_t}{\tilde Z^*_t}= -\frac{\dot{P_t}}{P_t} (dt + \frac{1}{\rho_t}dW_t),~\tilde Z^*_0 = -\frac{1}{2\mu}.$$
Therefore, we obtain
$$\bbE[\tilde Z^*_T]= -\frac{1}{2\mu}\frac{P_0}{P_T}.$$
Using Theorem \ref{theoxyz}, we get
\begin{equation*}
\bbE\Big[(\tilde Z^*_T)^2 + \frac{1}{2}\int_0^T(s_t^* Y_t\tilde\sigma_t)^2 dt\Big] =\big(\frac{1}{2\mu}\big)^2\frac{P_0}{P_T}.
\end{equation*}
Consequently, we have that the optimal variable $Z_{T,s^*}$ satisfies
$$\bbE[Z_{T,s^*}]=\frac{1}{2\mu}\big(1-\frac{P_0}{P_T}\big)$$ and
$$\bbE\big[\big(Z_{T,s^*}-\frac{1}{2\mu}\big)^2\big]=\big(\frac{1}{2\mu}\big)^2\frac{P_0}{P_T}.$$ Hence
$$\bbE[(Z_{T,s^*})^2]=\big(\frac{1}{2\mu}\big)^2(1-\frac{P_0}{P_T}).$$
We have thus proved the following proposition.
\begin{prop}
In the Black-Scholes model with time varying coefficients, the efficient points are the couples of the form
$$(m,m^2\frac{P_T}{P_T-P_0}),$$ with $m>0$ (remark that the ratio
$\frac{P_T}{P_T-P_0}$ does not depend on $\mu$).
Furthermore, the associated process $s_t^*$ enabling to compute optimal rules according to Theorem \ref{theo : main result} is explicitly given by 
\begin{equation*}
\frac{1}{3}s^*_tY_t = -\frac{1}{b_t}\frac{\dot{P_t}}{P_t}\tilde Z^*_t,
\end{equation*}
with
$$\frac{d\tilde Z^*_t}{\tilde Z^*_t}= -\frac{1}{b_t}\frac{\dot{P_t}}{P_t} \frac{dY_t}{Y_t},~\tilde Z^*_0 = -\frac{1}{2\mu}.$$

\end{prop}

\noindent Note that in practice, $\tilde Z^*$ is not observable. However, it can of course be approximated by a process $\tilde Z^{(*)}$ thanks to historical data, using for example a scheme of the form
\begin{equation*}
\tilde Z^{(*)}_{t_{i+1}} = \tilde Z^{(*)}_{t_i}\Big(1 -\frac{1}{b_{t_i}}\frac{\dot{P}_{t_i}}{P_{t_i}}\frac{Y_{t_{i+1}}-Y_{t_i}}{Y_{t_i}}\Big), \quad \tilde Z^{(*)}_0 = -\frac{1}{2\mu},
\end{equation*}
where the $t_i$ are the observation times of market data.

\newpage
\appendixtitleon
\appendixtitletocon
\begin{appendices}
\section{Proofs}
In the following $C$ denotes a constant which may vary from line to line. Note that we use several localization procedures in the proofs. We often give them in details since some of them are slightly unusual, in particular because of the fact that $\sigma^X$ may vanish at maturity.

\subsection{Proof of Proposition~\ref{proposition: stable conv}}

\noindent We start by proving in a very standard way the stable convergence of $\varepsilon_n^{-1}Z^n$ in $C[0,T]$, which is stronger than the weak convergence. More precisely, we show that
for any bounded continuous function $f$ on $C[0,T]$ and bounded random variable $U$ defined on $(\Omega, \calF, \bbF, \bbP)$,
\begin{equation*}
\lim_{n \to \infty} \bbE[Uf(\varepsilon_n^{-1}Z^n_{\cdot})] =  \bbE[Uf(Z^\ast_{\cdot})],
\end{equation*}
where $Z^*$ is defined by
\begin{equation*}
 Z^\ast_t = 
\frac{1}{3}\int_0^t s_udY_u + \frac{1}{\sqrt{6}}\int_0^t\paren{a_u^2 - \frac{2}{3}s_u^2}^{1/2}\sigma^Y_udB_u,
\end{equation*}
on an extension of $(\Omega, \calF, \bbF, \bbP)$ on which $B$ is a Brownian motion independent of all the other quantities. 
For $K>0$, we set
\begin{equation*}
 \alpha^K = \inf\{t > 0; |\rho_t| \geq K\} \wedge T.
\end{equation*}
Since $\rho$ is continuous on $[0,T]$ almost surely, 
\begin{equation*}
 \lim_{K \to \infty}\bbP [\alpha^K < T] = 0.
\end{equation*}
Now remark that
\begin{equation*}
\begin{split}
&| \bbE[Uf(\varepsilon_n^{-1}Z^n_{\cdot})] - \bbE[Uf(Z^\ast_{\cdot})]| \\
&\leq |\bbE[Uf(\varepsilon_n^{-1}Z^n_{\cdot})] -
 \bbE[Uf(\varepsilon_n^{-1}Z^n_{\cdot \wedge \alpha^K})]| 
+ |\bbE[Uf(\varepsilon_n^{-1}Z^n_{\cdot \wedge \alpha^K})]
- \bbE[Uf(Z^\ast_{\cdot \wedge \alpha^K})]| 
\\ & \hspace*{1cm}
+  |\bbE[Uf(Z^\ast_{\cdot \wedge \alpha^K})]
- \bbE[Uf(Z^\ast_{\cdot})]| 
\\ &\leq 4 \|f\|_\infty \|U\|_\infty\bbP[\alpha^K < T]
+ |\bbE[Uf(\varepsilon_n^{-1}Z^n_{\cdot \wedge \alpha^K})]
- \bbE[Uf(Z^\ast_{\cdot \wedge \alpha^K})]|.
\end{split}
\end{equation*}
Consequently, it suffices to show that for any $K>0$,
\begin{equation*}
\lim_{n \to \infty}  |\bbE[Uf(\varepsilon_n^{-1}Z^n_{\cdot \wedge \alpha^K})
- \bbE[Uf(Z^\ast_{\cdot \wedge \alpha^K})]| = 0.
\end{equation*}
Let 
\begin{equation*}
 \mathcal{E} = \exp\big\{
-\int_0^{\alpha^K} \rho_t dW^Y_t - \frac{1}{2}\int_0^{\alpha^K} \rho_t^2 dt 
\big\}.
\end{equation*}
Since $\bbE[\mathcal{E}] = 1$, the measure $\bbQ$ defined by
\begin{equation*}
 \frac{d\bbQ}{d\bbP} = \mathcal{E}
\end{equation*}
is a probability measure under which $Z^n_{\cdot \wedge \alpha^K}$ is a local martingale.
Under $\bbQ$, the uniform convergences in probability \eqref{asmp: limit 2} and 
\eqref{asmp: limit 3} on $[0,T]$ remain true.
Therefore by Theorem~IX.7.3 of \cite{jacod2003limit}, we have the stable convergence of $\varepsilon_n^{-1}Z^n_{\cdot \wedge \alpha^K}$
to $Z^\ast_{\cdot \wedge \alpha^K}$ under $\bbQ$.
Note that $\tilde{U} = U/\mathcal{E}$ is a $\bbQ$-integrable positive random variable and moreover, for all $A > 0$,
\begin{equation*}
\begin{split}
&|\bbE[Uf(\varepsilon_n^{-1}Z^n_{\cdot \wedge \alpha^K})]
- \bbE[Uf(Z^\ast_{\cdot \wedge \alpha^K})]| \\
\leq 
&|\bbE^\bbQ[\tilde{U}f(\varepsilon_n^{-1}Z^n_{\cdot \wedge \alpha^K})]
- \bbE^\bbQ[\tilde{U}f(Z^\ast_{\cdot \wedge \alpha^K})]| 
\\\leq &
|\bbE^\bbQ[(\tilde{U} \wedge A)f(\varepsilon_n^{-1}Z^n_{\cdot \wedge \alpha^K})]
- \bbE^\bbQ[(\tilde{U} \wedge A)f(Z^\ast_{\cdot \wedge \alpha^K})]| 
+ 2 \|f\|_\infty \bbE^{\bbQ}[\tilde{U}\mathbbm{1}_{\{\tilde{U} \geq A\}}].
\end{split}
\end{equation*}
The second term tends to $0$ uniformly in $n$ as $A \to \infty$.
The first term converges to $0$ due to the stable convergence under $\bbQ$ since $(\tilde{U} \wedge A)$ is a bounded random variable.
\\

\noindent
Now, we prove $a^2 \geq s^2$ under the additional condition \eqref{Condition add}.
Since $a$ and $s$ are continuous, it suffices to show
$a_t^2 \geq s_t^2$ for all $t \in [0,T)$.
Fix $T_0 < T$ and let
\begin{equation} \label{alphahat}
 \hat{\alpha}^K = \inf\{u >0; |b^X_u|\vee \sigma^X_u \vee \sigma^Y_u \geq K \text{ or } \sigma^X_u \leq 1/K\} \wedge T_0
\end{equation}
for $K > 0$.
Since $\sigma^X$ is positive and continuous on $[0,T_0]$, we have
\begin{equation} \label{alphahat conv}
 \lim_{K\to \infty}\bbP[\hat{\alpha}^K < T_0] = 0.
\end{equation}
Therefore, it suffices to show
\begin{equation}\label{eq: local a s}
 a_{u \wedge \hat{\alpha}^K}^2 \geq s^2_{u \wedge \hat{\alpha}^K}
\end{equation}
for all $u \geq 0$ and $K > 0$. 
Fix $K$ and define the probability measure $\hat{\bbQ}$ by

\begin{equation*}
 \frac{d\hat{\bbQ}}{d \bbP} = \exp\Big\{- \int_0^{\hat{\alpha}^K} \frac{b^X_u}{\sigma^X_u}d W^X_u - \frac{1}{2}\int_0^{\hat{\alpha}^K}\big(\frac{b^X_u}{\sigma^X_u} \big)^2 d u \Big\}.
\end{equation*}
Under $\hat{\bbQ}$, $X_{\cdot \wedge \hat{\alpha}^K}$ is a martingale with bounded quadratic variation.
Since $\hat{\bbQ}$ is equivalent to $\bbP$, it suffices to show \eqref{eq: local a s} under $\hat{\bbQ}$. \\

\noindent
By \eqref{Condition add}, there exists a subsequence $\{n(k)\}$ such that
\begin{equation*}
 \hat{\bbQ} \Big[\varepsilon_{n(k)}^{-4/3}\sup_{ j\geq 0}(\tau^{n(k)}_{j+1} \wedge
 T_0 - \tau^{n(k)}_j \wedge T_0) > \frac{1}{k} \Big] < \frac{1}{k}.
\end{equation*}
Let
\begin{equation*}
T_k = \inf\big\{u > 0, \varepsilon_{n(k)}^{-4/3}\sup_{ j\geq 0}(\tau^{n(k)}_{j+1} \wedge
 u - \tau^{n(k)}_j \wedge
 u) > \frac{1}{k}\big\} \wedge \hat{\alpha}^K.
\end{equation*}
Then
\begin{equation*}
\lim_{k \to \infty}\hat{\bbQ}[T_k < \hat{\alpha}^K] = 0
\end{equation*}
and so,
\begin{equation} \label{loc conv}
 \begin{split}
  & \varepsilon_{n(k)}^{-1}
\crochet{Z^{n(k)},Y}_{t \wedge T_k} \to \frac{1}{3}\int_0^{t \wedge \hat{\alpha}^K}
s_u (\sigma^Y_u)^2 du,\\
  &  \varepsilon_{n(k)}^{-2}
\crochet{Z^{n(k)}}_{t \wedge T_k} \to \frac{1}{6}\int_0^{t \wedge \hat{\alpha}^K}
a_u^2 (\sigma^Y_u)^2 du,
 \end{split}
\end{equation}
in probability as $k \to \infty$ for all $ t\geq 0$.
Let
\begin{equation*}
 \hat{\tau}^k_j = \tau^{n(k)}_j  \wedge T_k
\end{equation*}
for $j \geq 0$.
We now give three technical lemmas.
\begin{lem}\label{lem1}
Let $\kappa_u=(\sigma^Y_u/\sigma^X_u)^2$. We have
\begin{equation*}
\frac{1}{3} \varepsilon_{n(k)}^{-1}\sum_{j=0}^{N^{n(k)}_{t \wedge T_k}}
\kappa_{\hat{\tau}^k_j \wedge t}
\bbE^{\hat{\bbQ}}\Big[
(X_{\hat{\tau}^k_{j+1} \wedge t} - 
X_{\hat{\tau}^k_j \wedge t}
)^3 \big| \mathcal{F}_{\hat{\tau}^k_j \wedge t}
\Big]
 -
\varepsilon_{n(k)}^{-1}\crochet{Z^{n(k)},Y}_{t \wedge T_k} \to 0,
\end{equation*}
in probability as $k\to \infty$ for all $t \geq 0$.
\end{lem}
\begin{proof}
By It$\hat{\text{o}}$'s formula,
\begin{equation*}
\frac{1}{3}
\bbE^{\hat{\bbQ}}\big[
(X_{\hat{\tau}^k_{j+1} \wedge t} - 
X_{\hat{\tau}^k_j \wedge t}
)^3 \big| \mathcal{F}_{\hat{\tau}^k_j \wedge t}
\big]
=
\mathbb{E}^{\hat{\bbQ}}\big[
\int_{\hat{\tau}^k_j \wedge t}^{\hat{\tau}^k_{j+1} \wedge t} (X_u - X_{\hat{\tau}^k_j})d \crochet{X}_u
 \big| \mathcal{F}_{\hat{\tau}^k_j \wedge t}
\big].
\end{equation*}
We now show that
\begin{equation}\label{lln1}
\begin{split}
& \varepsilon_{n(k)}^{-1}\sum_{j=0}^{N^{n(k)}_{t\wedge T_k}}
\kappa_{\hat{\tau}^k_j \wedge t}
\mathbb{E}^{\hat{\bbQ}}\big[
\int_{\hat{\tau}^k_j \wedge t}^{\hat{\tau}^k_{j+1} \wedge t} (X_u - X_{\hat{\tau}^k_j})d \crochet{X}_u
 \big| \mathcal{F}_{\hat{\tau}^k_j \wedge t}
\big]
\\ & -
 \varepsilon_{n(k)}^{-1}\sum_{j=0}^{N^{n(k)}_{t\wedge T_k}}
\kappa_{\hat{\tau}^k_j \wedge t}
\int_{\hat{\tau}^k_j \wedge t}^{\hat{\tau}^k_{j+1} \wedge t} (X_u - X_{\hat{\tau}^k_j})d \crochet{X}_u
\to 0
\end{split}
\end{equation}
and
\begin{equation}\label{disc kappa}
 \varepsilon_{n(k)}^{-1}\sum_{j=0}^{N^{n(k)}_{t\wedge T_k}}
\kappa_{\hat{\tau}^k_j \wedge t}
\int_{\hat{\tau}^k_j \wedge t}^{\hat{\tau}^k_{j+1} \wedge t} (X_u - X_{\hat{\tau}^k_j})d \crochet{X}_u
-  \varepsilon_{n(k)}^{-1}\crochet{Z^{n(k)},Y}_{t \wedge T_k} \to 0, 
\end{equation}
in probability. \\

\noindent
By Lenglart inequality for discrete martingales (see e.g., Lemma~A.2 of \cite{fukasawa2011discretization}), a sufficient condition for (\ref{lln1}) is the fact that
\begin{equation}\label{lln1qv}
 \varepsilon_{n(k)}^{-2}\sum_{j=0}^{N^{n(k)}_{t\wedge T_k}}
\kappa_{\hat{\tau}^k_j \wedge t}^2
\mathbb{E}^{\hat{\bbQ}}\big[ \big(
\int_{\hat{\tau}^k_j \wedge t}^{\hat{\tau}^k_{j+1} \wedge t} (X_u - X_{\hat{\tau}^k_j})d \crochet{X}_u
\big)^2 \big| \mathcal{F}_{\hat{\tau}^k_j \wedge t}
\big] \to 0,
\end{equation}
in probability. To get this convergence, first use successively H\"older inequality, It\^o's formula and Burkholder-Davis-Gundy inequality to obtain that
\begin{equation*}
\begin{split}
&\sum_{j=0}^{N^{n(k)}_{t \wedge T_k}}
\kappa_{\hat{\tau}^k_j \wedge t}^2
\bbE^{\hat{\bbQ}}\big[
\big(
\int_{\hat{\tau}^k_j \wedge t}^{\hat{\tau}^k_{j+1} \wedge t}
(X_u-X_{\hat{\tau}^k_j})d\langle X
 \rangle_u
\big)^2 \big| \mathcal{F}_{\hat{\tau}^k_j \wedge t}
\big] 
\\
& \leq C\sum_{j=0}^{N^{n(k)}_{t \wedge T_k}}
\kappa_{\hat{\tau}^k_j \wedge t}^2
\bbE^{\hat{\bbQ}}\Big[\big
(\langle X \rangle_{\hat{\tau}^k_{j+1} \wedge t} - \langle X
 \rangle_{\hat{\tau}^k_j \wedge t}\big)^{3/2}
\big(
\int_{\hat{\tau}^k_j \wedge t}^{\hat{\tau}^k_{j+1} \wedge t}
(X_u-X_{\hat{\tau}^k_j})^4d\langle X
 \rangle_u
\big)^{1/2} \big| \mathcal{F}_{\hat{\tau}^k_j \wedge t}
\Big] \\
& \leq C
\Big(\sum_{j=0}^{N^{n(k)}_{t \wedge T_k}}
\bbE^{\hat{\bbQ}}\big[
\big(\langle X \rangle_{\hat{\tau}^k_{j+1} \wedge t} - \langle X
 \rangle_{\hat{\tau}^k_j \wedge t}\big)^3
\big| \mathcal{F}_{\hat{\tau}^k_j \wedge t}
\big]\Big)^{1/2}
\Big(\sum_{j=0}^{N^{n(k)}_{t \wedge T_k}}
\bbE^{\hat{\bbQ}}\big[
\int_{\hat{\tau}^k_j \wedge t}^{\hat{\tau}^k_{j+1} \wedge t}
(X_u-X_{\hat{\tau}^k_j})^4d\langle X
 \rangle_u
 \big| \mathcal{F}_{\hat{\tau}^k_j \wedge t}
\big]\Big)^{1/2}
\\
& = C
\Big(\sum_{j=0}^{N^{n(k)}_{t \wedge T_k}}
\bbE^{\hat{\bbQ}}\big[
\big(\langle X \rangle_{\hat{\tau}^k_{j+1} \wedge t} - \langle X
 \rangle_{\hat{\tau}^k_j \wedge t}\big)^3
\big| \mathcal{F}_{\hat{\tau}^k_j \wedge t}
\big]\Big)^{1/2}
\Big(\sum_{j=0}^{N^{n(k)}_{t \wedge T_k}}
\bbE^{\hat{\bbQ}}\big[
(X_{\hat{\tau}^k_{j+1}\wedge t}-X_{\hat{\tau}^k_j \wedge t})^6
 | \mathcal{F}_{\hat{\tau}^k_j \wedge t}
\big]\Big)^{1/2} 
\\
& \leq
C\sum_{j=0}^{N^{n(k)}_{t \wedge T_k}}
\bbE^{\hat{\bbQ}}\big[
\big(\langle X \rangle_{\hat{\tau}^k_{j+1} \wedge t} - \langle X
 \rangle_{\hat{\tau}^k_j \wedge t}\big)^3
\big| \mathcal{F}_{\hat{\tau}^k_j \wedge t}
\big] \\
& \leq
C\sum_{j=0}^{N^{n(k)}_{t \wedge T_k}}
\bbE^{\hat{\bbQ}}\big[
|\hat{\tau}^k_{j+1} \wedge t - 
\hat{\tau}^k_j \wedge t|^3
\big| \mathcal{F}_{\hat{\tau}^k_j \wedge t}
\big]. 
\end{split}
\end{equation*}
Note also that
\begin{equation*}
 \left\{
j \leq N^{n(k)}_{t \wedge T_k} \right\} = \left\{
\tau^{n(k)}_j \leq t \wedge T_k
\right\} \in \mathcal{F}_{\hat{\tau}^k_j \wedge t}.
\end{equation*}
Then (\ref{lln1qv}) follows since
\begin{equation*}
 \varepsilon_{n(k)}^{-2}\bbE^{\hat{\bbQ}}\Big[
\sum_{j=0}^{\infty}
|\hat{\tau}^k_{j+1} \wedge t  - \hat{\tau}^k_j \wedge t|^3
\Big]
\leq
 \varepsilon_{n(k)}^{-2}
\frac{\varepsilon_{n(k)}^{8/3}}{k^2}
\bbE^{\hat{\bbQ}}\Big[
\sum_{j=0}^{\infty}
|\hat{\tau}^k_{j+1} \wedge t  - \hat{\tau}^k_j \wedge t|
\Big]
\leq \frac{\varepsilon_{n(k)}^{2/3}}{k^2} t \to 0.
\end{equation*}

\noindent
We now turn to \eqref{disc kappa}. Note that
\begin{equation*}
  \varepsilon_{n(k)}^{-1}\sum_{j=0}^{N^{n(k)}_{t\wedge T_k}}
\kappa_{\hat{\tau}^k_j \wedge t}
\int_{\hat{\tau}^k_j \wedge t}^{\hat{\tau}^k_{j+1} \wedge t} (X_u - X_{\hat{\tau}^k_j})d \crochet{X}_u
=
  \varepsilon_{n(k)}^{-1}\sum_{j=0}^\infty
\kappa_{\hat{\tau}^k_j \wedge t}
\int_{\hat{\tau}^k_j \wedge t}^{\hat{\tau}^k_{j+1} \wedge t} (X_u - X_{\hat{\tau}^k_j})d \crochet{X}_u
\end{equation*}
and
\begin{equation*}
 \varepsilon_{n(k)}^{-1}\crochet{Z^{n(k)},Y}_{t \wedge T_k} = 
   \varepsilon_{n(k)}^{-1}\sum_{j=0}^\infty
\int_{\hat{\tau}^k_j \wedge t}^{\hat{\tau}^k_{j+1} \wedge t} (X_u - X_{\hat{\tau}^k_j})\kappa_u d \crochet{X}_u.
\end{equation*}
Therefore, the absolute value of the left hand side of \eqref{disc
  kappa} is dominated by
\begin{align*}
&\varepsilon_{n(k)}^{-1}\sum_{j=0}^\infty
\int_{\hat{\tau}^k_j \wedge t}^{\hat{\tau}^k_{j+1} \wedge t} |X_u -
X_{\hat{\tau}^k_j}|\,|\kappa_u-\kappa_{\hat{\tau}^k_j \wedge t}| d
\crochet{X}_u\\
&\leq \Big(\varepsilon_{n(k)}^{-2}\sum_{j=0}^\infty
\int_{\hat{\tau}^k_j \wedge t}^{\hat{\tau}^k_{j+1} \wedge t} |X_u -
X_{\hat{\tau}^k_j}|^2\,|\kappa_u-\kappa_{\hat{\tau}^k_j \wedge t}|^2 d
\crochet{X}_u\Big)^{1/2}\Big(\sum_{j=0}^\infty
\int_{\hat{\tau}^k_j \wedge t}^{\hat{\tau}^k_{j+1} \wedge t} d
\crochet{X}_u\Big)^{1/2}\\
&\leq \sup_{u \geq 0}\frac{1}{\kappa_{u \wedge T_k}}
 \sup_{u \geq 0, j \geq 0}\big|
\kappa_{\hat{\tau}^k_{j+1} \wedge u} - \kappa_{ \hat{\tau}^k_j \wedge u}
\big|
 \varepsilon_{n(k)}^{-1}\crochet{Z^{n(k)}}^{1/2}_{t \wedge T_k}\crochet{X}^{1/2}_{t \wedge T_k},
\end{align*}
which converges to $0$ due to \eqref{loc conv} and the uniform continuity of $\kappa$.
\end{proof}

\begin{lem}\label{lem2}
We have
\begin{equation*}
\frac{1}{6} \varepsilon_{n(k)}^{-2}\sum_{j=0}^{N^{n(k)}_{t \wedge T_k}}
\kappa_{\hat{\tau}^k_j \wedge t}
\bbE^{\hat{\bbQ}}\big[
\big(X_{\hat{\tau}^k_{j+1} \wedge t} - 
X_{\hat{\tau}^k_j \wedge t}
\big)^4 \big| \mathcal{F}_{\hat{\tau}^k_j \wedge t}
\big]
 -
\varepsilon_{n(k)}^{-2}\crochet{Z^{n(k)}}_{t \wedge T_k} \to 0,
\end{equation*}
in probability as $k\to \infty$, for all $t \geq 0$.
\end{lem}
\begin{proof}
The proof is very similar to the previous one.
By It$\hat{\text{o}}$'s formula,
\begin{equation*}
\frac{1}{6}
\bbE^{\hat{\bbQ}}\big[
\big(X_{\hat{\tau}^k_{j+1} \wedge t} - 
X_{\hat{\tau}^k_j \wedge t}
\big)^4 \big| \mathcal{F}_{\hat{\tau}^k_j \wedge t}
\big]
=
\mathbb{E}^{\hat{\bbQ}}\big[
\int_{\hat{\tau}^k_j \wedge t}^{\hat{\tau}^k_{j+1} \wedge t} (X_u - X_{\hat{\tau}^k_j})^2 d \crochet{X}_u
 \big| \mathcal{F}_{\hat{\tau}^k_j \wedge t}
\big].
\end{equation*}
We now show that
\begin{equation}\label{lln2}
\begin{split}
& \varepsilon_{n(k)}^{-2}\sum_{j=0}^{N^{n(k)}_{t\wedge T_k}}
\kappa_{\hat{\tau}^k_j \wedge t}
\mathbb{E}^{\hat{\bbQ}}\big[
\int_{\hat{\tau}^k_j \wedge t}^{\hat{\tau}^k_{j+1} \wedge t} (X_u - X_{\hat{\tau}^k_j})^2d \crochet{X}_u
 \big| \mathcal{F}_{\hat{\tau}^k_j \wedge t}
\big]
\\ & -
 \varepsilon_{n(k)}^{-2}\sum_{j=0}^{N^{n(k)}_{t\wedge T_k}}
\kappa_{\hat{\tau}^k_j \wedge t}
\int_{\hat{\tau}^k_j \wedge t}^{\hat{\tau}^k_{j+1} \wedge t} (X_u - X_{\hat{\tau}^k_j})^2d \crochet{X}_u
\to 0
\end{split}
\end{equation}
and
\begin{equation}\label{disc kappa2}
 \varepsilon_{n(k)}^{-2}\sum_{j=0}^{N^{n(k)}_{t\wedge T_k}}
\kappa_{\hat{\tau}^k_j \wedge t}
\int_{\hat{\tau}^k_j \wedge t}^{\hat{\tau}^k_{j+1} \wedge t} (X_u - X_{\hat{\tau}^k_j})^2d \crochet{X}_u
-  \varepsilon_{n(k)}^{-2}\crochet{Z^{n(k)}}_{t \wedge T_k} \to 0, 
\end{equation}
in probability. \\

\noindent
By Lenglart inequality for discrete martingales, a sufficient condition for (\ref{lln2}) is
\begin{equation}\label{lln2qv}
 \varepsilon_{n(k)}^{-4}\sum_{j=0}^{N^{n(k)}_{t\wedge T_k}}
\kappa_{\hat{\tau}^k_j \wedge t}^2
\mathbb{E}^{\hat{\bbQ}}\big[ \Big(
\int_{\hat{\tau}^k_j \wedge t}^{\hat{\tau}^k_{j+1} \wedge t} (X_u - X_{\hat{\tau}^k_j})^2d \crochet{X}_u
\Big)^2 \big| \mathcal{F}_{\hat{\tau}^k_j \wedge t}
\big] \to 0,
\end{equation}
in probability. To get this convergence, first use successively H\"older inequality, It\^o's formula and Burkholder-Davis-Gundy inequality to obtain that
\begin{equation*}
\begin{split}
&\sum_{j=0}^{N^{n(k)}_{t \wedge T_k}}
\kappa_{\hat{\tau}^k_j \wedge t}^2
\bbE^{\hat{\bbQ}}\Big[
\big(
\int_{\hat{\tau}^k_j \wedge t}^{\hat{\tau}^k_{j+1} \wedge t}
(X_u-X_{\hat{\tau}^k_j})^2d\langle X
 \rangle_u
\big)^2 \big| \mathcal{F}_{\hat{\tau}^k_j \wedge t}
\Big] 
\\
& \leq
\sum_{j=0}^{N^{n(k)}_{t \wedge T_k}}
\kappa_{\hat{\tau}^k_j \wedge t}^2
\bbE^{\hat{\bbQ}}\Big[
\big(\langle X \rangle_{\hat{\tau}^k_{j+1} \wedge t} - \langle X
 \rangle_{\hat{\tau}^k_j \wedge t}\big)^{4/3}
\big(
\int_{\hat{\tau}^k_j \wedge t}^{\hat{\tau}^k_{j+1} \wedge t}
(X_u-X_{\hat{\tau}^k_j})^6d\langle X
 \rangle_u
\big)^{2/3} \big| \mathcal{F}_{\hat{\tau}^k_j \wedge t}
\Big] \\
& \leq C
\Big(\sum_{j=0}^{N^{n(k)}_{t \wedge T_k}}
\bbE^{\hat{\bbQ}}\Big[
\big(\langle X \rangle_{\hat{\tau}^k_{j+1} \wedge t} - \langle X
 \rangle_{\hat{\tau}^k_j \wedge t}\big)^4
\big| \mathcal{F}_{\hat{\tau}^k_j \wedge t}
\Big]\Big)^{1/3}
\Big(\sum_{j=0}^{N^{n(k)}_{t \wedge T_k}}
\bbE^{\hat{\bbQ}}\Big[
\int_{\hat{\tau}^k_j \wedge t}^{\hat{\tau}^k_{j+1} \wedge t}
(X_u-X_{\hat{\tau}^k_j})^6d\langle X
 \rangle_u
 \big| \mathcal{F}_{\hat{\tau}^k_j \wedge t}
\Big]\Big)^{2/3}
\\
& = C
\Big(\sum_{j=0}^{N^{n(k)}_{t \wedge T_k}}
\bbE^{\hat{\bbQ}}\Big[
\big(\langle X \rangle_{\hat{\tau}^k_{j+1} \wedge t} - \langle X
 \rangle_{\hat{\tau}^k_j \wedge t}\big)^4
\big| \mathcal{F}_{\hat{\tau}^k_j \wedge t}
\Big]\Big)^{1/3}
\Big(\sum_{j=0}^{N^{n(k)}_{t \wedge T_k}}
\bbE^{\hat{\bbQ}}\Big[
(X_{\hat{\tau}^k_{j+1}\wedge t}-X_{\hat{\tau}^k_j \wedge t})^8
 | \mathcal{F}_{\hat{\tau}^k_j \wedge t}
\Big]\Big)^{2/3}
\\
& \leq
C \sum_{j=0}^{N^{n(k)}_{t \wedge T_k}}
\bbE^{\hat{\bbQ}}\big[
\big(\langle X \rangle_{\hat{\tau}^k_{j+1} \wedge t} - \langle X
 \rangle_{\hat{\tau}^k_j \wedge t}\big)^4
\big| \mathcal{F}_{\hat{\tau}^k_j \wedge t}
\big] \\
& \leq 
C \sum_{j=0}^{N^{n(k)}_{t \wedge T_k}}
\bbE^{\hat{\bbQ}}\big[
|\hat{\tau}^k_{j+1} \wedge t - \hat{\tau}^k_j \wedge t|^4
\big| \mathcal{F}_{\hat{\tau}^k_j \wedge t}
\big].
\end{split}
\end{equation*}
Then, observe that
\begin{equation*}
\varepsilon_{n(k)}^{-4}\bbE^{\hat{\bbQ}} \Big[
\sum_{j=0}^\infty
|\hat{\tau}^k_{j+1} \wedge t - 
\hat{\tau}^k_j \wedge t|^4\Big]
\leq \frac{t}{k^3} \to 0,
\end{equation*}
which gives (\ref{lln2}). The proof for \eqref{disc kappa2} is obtained in the same way as that for \eqref{disc kappa}.
\end{proof}
\noindent We finally give the following almost straightforward result, which is easily deduced from simplified versions of the proofs of the previous lemma. 
\begin{lem}\label{lem_additional}
We have
\begin{equation*}
\sum_{j=0}^{N^{n(k)}_{t \wedge T_k}}
\kappa_{\hat{\tau}^k_j \wedge t}
\bbE^{\hat{\bbQ}}\big[
\big(X_{\hat{\tau}^k_{j+1} \wedge t} - 
X_{\hat{\tau}^k_j \wedge t}
\big)^2 \big| \mathcal{F}_{\hat{\tau}^k_j \wedge t}
\big]
 -
\crochet{Y}_{t \wedge T_k} \to 0,
\end{equation*}
in probability as $k\to \infty$ for all $t \geq 0$.
\end{lem}

\noindent
We are now ready to complete the proof of Proposition~\ref{proposition: stable conv}.
From \eqref{loc conv} and Lemmas~\ref{lem1}, \ref{lem2} and \ref{lem_additional}, we have for all $0\leq v\leq t$ the following convergences in probability as $k\to \infty$:
\begin{equation*}
\begin{split}
& \sum_{j=N^{n(k)}_{v \wedge T_k}+1}^{N^{n(k)}_{t \wedge T_k}}
\kappa_{\hat{\tau}^k_j \wedge t}
\bbE^{\hat{\bbQ}}\big[
\big(X_{\hat{\tau}^k_{j+1} \wedge t} - 
X_{\hat{\tau}^k_j \wedge t}
\big)^2 \big| \mathcal{F}_{\hat{\tau}^k_j \wedge t}
\big] \to 
\int_{v \wedge \hat{\alpha}^K}^{t \wedge \hat{\alpha}^K}  (\sigma^Y_y)^2 du, \\
& \varepsilon_{n(k)}^{-1}\sum_{j=N^{n(k)}_{v \wedge T_k}+1}^{N^{n(k)}_{t \wedge T_k}}
\kappa_{\hat{\tau}^k_j \wedge t}
\bbE^{\hat{\bbQ}}\big[
\big(X_{\hat{\tau}^k_{j+1} \wedge t} - 
X_{\hat{\tau}^k_j \wedge t}
\big)^3 \big| \mathcal{F}_{\hat{\tau}^k_j \wedge t}
\big] \to 
\int_{v \wedge \hat{\alpha}^K}^{t \wedge \hat{\alpha}^K} s_u (\sigma^Y_y)^2 du, \\
& \varepsilon_{n(k)}^{-2}\sum_{j=N^{n(k)}_{v \wedge T_k}+1}^{N^{n(k)}_{t \wedge T_k}}
\kappa_{\hat{\tau}^k_j \wedge t}
\bbE^{\hat{\bbQ}}\big[
\big(X_{\hat{\tau}^k_{j+1} \wedge t} - 
X_{\hat{\tau}^k_j \wedge t}
\big)^4 \big| \mathcal{F}_{\hat{\tau}^k_j \wedge t}
\big] \to 
\int_{v \wedge \hat{\alpha}^K}^{t \wedge \hat{\alpha}^K} a_u^2 (\sigma^Y_y)^2 du. 
 \end{split}
\end{equation*}
Since
\begin{equation*}
\big(
\bbE^{\hat{\bbQ}}\big[
(X_{\hat{\tau}^k_{j+1}\wedge t}-X_{\hat{\tau}^k_j \wedge t})^3
 \big| \mathcal{F}_{\hat{\tau}^k_j \wedge t}
\big]\big)^2
\leq 
\bbE^{\hat{\bbQ}}\big[
(X_{\hat{\tau}^k_{j+1}\wedge t}-X_{\hat{\tau}^k_j \wedge t})^2
 \big| \mathcal{F}_{\hat{\tau}^k_j \wedge t}
\big]
\bbE^{\hat{\bbQ}}\big[
(X_{\hat{\tau}^k_{j+1}\wedge t}-X_{\hat{\tau}^k_j \wedge t})^4
 \big| \mathcal{F}_{\hat{\tau}^k_j \wedge t}
\big],
\end{equation*}
we have
\begin{equation*}
\begin{split}
&\Big(\varepsilon_{n(k)}^{-1}\sum_{j=N^{n(k)}_{v \wedge T_k}+1}^{N^{n(k)}_{t \wedge T_k}}
\kappa_{\hat{\tau}^k_j}
\bbE^{\hat{\bbQ}}\big[
(X_{\hat{\tau}^k_{j+1}\wedge t}-X_{\hat{\tau}^k_j \wedge t})^3
 \big| \mathcal{F}_{\hat{\tau}^k_j \wedge t}
\big]\Big)^2
\\ &\leq
\varepsilon_{n(k)}^{-2}\sum_{j=N^{n(k)}_{v \wedge T_k}+1}^{N^{n(k)}_{t \wedge T_k}}
\kappa_{\hat{\tau}^k_j}
\bbE^{\hat{\bbQ}}\left[
(X_{\hat{\tau}^k_{j+1}\wedge t}-X_{\hat{\tau}^k_j \wedge t})^4
 \big| \mathcal{F}_{\hat{\tau}^k_j \wedge t}
\right]
\sum_{j=N^{n(k)}_{v \wedge T_k}+1}^{N^{n(k)}_{t \wedge T_k}}
\kappa_{\hat{\tau}^k_j}
\bbE^{\hat{\bbQ}}\left[
(X_{\hat{\tau}^k_{j+1}\wedge t}-X_{\hat{\tau}^k_j \wedge t})^2
 \big| \mathcal{F}_{\hat{\tau}^k_j \wedge t}
\right].
\end{split}
\end{equation*}
This implies that for all $0\leq v\leq t$,
\begin{equation*}
\big(
\int_{v \wedge \hat{\alpha}^K}^{t \wedge \hat{\alpha}^K} s_u(\sigma^Y_u)^2du
\big)^2 \leq
\int_{v \wedge \hat{\alpha}^K}^{t \wedge \hat{\alpha}^K} a_u^2(\sigma^Y_u)^2du
\int_{v \wedge \hat{\alpha}^K}^{t \wedge \hat{\alpha}^K} (\sigma^Y_u)^2 du.
\end{equation*}
Thus we obtain \eqref{eq: local a s}.

\subsection{Proof of Proposition \ref{prop: intuition a s}}
In this proof, using a classical localization procedure together with Girsanov theorem, we can assume that $b^X=0$ and that $\sigma^X$ and $\sigma^Y$ are bounded on $[0,T]$. We start with two technical lemmas and their proof.\\

\begin{lem}\label{lemma: conv diff qv}
For any $p \in [0,2)$,
\begin{equation*}
 \varepsilon_n^{-p} \sup_{j \geq 0}(\crochet{X}_{\tau^n_{j+1}}  -\crochet{X}_{\tau^n_j}) \to 0,
\end{equation*}
in probability. 
\end{lem}
\begin{proof}
Let $K > 0$ and
\begin{equation} \label{loc gamma}
 \gamma^n_K = \inf\{t > 0; \varepsilon_n^{-1}|X_t-X^n_t| \geq K \} \wedge T.
\end{equation}
Using the tightness of the family \eqref{sup diff}, we get
\begin{equation} \label{loc gamma unif}
 \lim_{K \to \infty} \sup_{n \in \mathbb{N}}\bbP[ \gamma^n_K < T] = 0.
\end{equation}
Therefore, it is enough to show that for any $K>0$,
\begin{equation*}
 \varepsilon_n^{-p} \sup_{j \geq 0}(\crochet{X}_{\tau^n_{j+1} \wedge \gamma^n_K}  -\crochet{X}_{\tau^n_j \wedge \gamma^n_K}) \to 0,
\end{equation*}
in probability. Take an integer $m > 2/(2-p)$. Since
\begin{equation*}
 \sup_{j \geq 0} 
(\crochet{X}_{\tau^n_{j+1} \wedge \gamma^n_K}  -\crochet{X}_{\tau^n_j \wedge \gamma^n_K})
  \leq \Big(
\sum_{j=0}^{\infty}
(\crochet{X}_{\tau^n_{j+1} \wedge \gamma^n_K}  -\crochet{X}_{\tau^n_j \wedge \gamma^n_K})^m 
\Big)^{1/m},
\end{equation*}
the statement of the lemma follows from the fact that
\begin{equation*}
\begin{split}
\bbE \Big[ \sum_{j=0}^{\infty}
(\crochet{X}_{\tau^n_{j+1} \wedge \gamma^n_K} - 
\crochet{X}_{\tau^n_j \wedge \gamma^n_K})^m \Big]
& \leq
C
\bbE \Big[ \sum_{j=0}^{\infty}
\big(\sup_{t \geq 0}|X_{\tau^n_{j+1} \wedge \gamma^n_K \wedge t} - 
X_{\tau^n_j \wedge \gamma^n_K \wedge t}|\big)^{2m} \Big]
 \\
& \leq
C \varepsilon_n^{2m-2}
\bbE \Big[ \sum_{j=0}^{\infty}
\big(\sup_{t \geq 0}|X_{\tau^n_{j+1} \wedge \gamma^n_K \wedge t} - 
X_{\tau^n_j \wedge \gamma^n_K \wedge t}|\big)^2 \Big]
 \\
& \leq C
\varepsilon_n^{2m-2}
\bbE \Big[ \sum_{j=0}^{\infty}
(\crochet{X}_{\tau^n_{j+1} \wedge \gamma^n_K} - 
\crochet{X}_{\tau^n_j \wedge \gamma^n_K}) \Big]
 \\
 & \leq C \varepsilon_n^{2m-2}.
\end{split}
\end{equation*}
Here we have used that $\bbE[\crochet{X}_T] < \infty$. The result follows using H\"older inequality.
\end{proof}

\begin{lem} \label{lemma: condition add}
For any $p \in [0,2)$ and $T_0 \in [0,T)$,
\begin{equation*}
 \varepsilon_n^{-p} \sup_{j \geq 0}(\tau^n_{j+1} \wedge T_0 -\tau^n_j \wedge T_0) \to 0,
\end{equation*}
in probability. In particular,
the convergence in probability \eqref{Condition add} holds for all $T_0 \in [0,T)$.
\end{lem}

\begin{proof}
 Let $T_0 \in [0,T)$, $K > 0$ and
\begin{equation} 
 \hat{\gamma}_K = \inf\{t > 0;  \sigma^X_t \leq 1/K\} \wedge T_0.
\end{equation}
Using the continuity and the positivity of $\sigma^X$ on $[0,T)$, we get
\begin{equation*}
 \lim_{K \to \infty} \bbP[ \hat{\gamma}_K < T_0] = 0.
\end{equation*}
Therefore, it is enough to show that for any $K>0$,
\begin{equation*}
 \varepsilon_n^{-p} \sup_{j \geq 0} (\tau^n_{j+1} \wedge \hat{\gamma}_K - 
\tau^n_j \wedge \hat{\gamma}_K) \to 0,
\end{equation*}
in probability. 
This follows from Lemma~\ref{lemma: conv diff qv} since
\begin{equation*}
\sup_{j \geq 0} (\tau^n_{j+1} \wedge \hat{\gamma}_K - 
\tau^n_j \wedge \hat{\gamma}_K)
\leq C  \sup_{j \geq 0}(\crochet{X}_{\tau^n_{j+1}}  -\crochet{X}_{\tau^n_j}).
\end{equation*}
\end{proof}

\noindent 
We now give the end of the proof of Proposition~\ref{prop: intuition a s}.
Define $\gamma^n_K$ by \eqref{loc gamma}.
It suffices to show that for any $K>0$
\begin{equation*}
 \begin{split}
&  \sup_{t\geq 0}
\Big|
\varepsilon_n^{-2}\crochet{Z^n}_{t \wedge \gamma^n_K} - 
\frac{1}{6}\int_0^{t \wedge \gamma^n_K}a_u^2(\sigma^Y_u)^2du
\Big| \to 0, \\
& \sup_{t\geq 0}
\Big|
\varepsilon_n^{-1}\crochet{Z^n, Y}_{t \wedge \gamma^n_K}
- \frac{1}{3}\int_0^{{t \wedge \gamma^n_K}}s_u(\sigma^Y_u)^2du 
\Big| \to 0,
 \end{split}
\end{equation*}
in probability.
Since 
\begin{equation} \label{unif bound}
\varepsilon_n^{-1} \sup_{t \geq 0} |X_{t\wedge \gamma^n_K}-
X^n_{t\wedge \gamma^n_K}| \leq K,
\end{equation}
the families $\varepsilon_n^{-2}\crochet{Z^n}_{\cdot \wedge \gamma^n_K}$ and
$\varepsilon_n^{-1}\crochet{Z^n,Y}_{\cdot \wedge \gamma^n_K}$ are equicontinuous.
So we just need to prove that for any $t \in [0,T)$,
\begin{equation*}
 \begin{split}
& \varepsilon_n^{-2}\crochet{Z^n}_{t \wedge \gamma^n_K} - 
\frac{1}{6}\int_0^{t \wedge \gamma^n_K}a_u^2(\sigma^Y_u)^2du \to 0, \\
& 
\varepsilon_n^{-1}\crochet{Z^n, Y}_{t \wedge \gamma^n_K}-  \frac{1}{3}\int_0^{{t \wedge \gamma^n_K}}s_u(\sigma^Y_u)^2du \to 0,
 \end{split}
\end{equation*}
in probability. Let
\begin{equation*}
 \beta^n_M = \inf\{u > 0 ; (1/\sigma^X_u) \geq M  \} \wedge t \wedge \gamma^n_K
\end{equation*}
for $M > 0$. Since $t < T$, Lemma~\ref{lemma: conv diff qv} gives that
\begin{equation*}
 \lim_{M\to \infty} \sup_{n \in \mathbb{N}}\bbP[\beta^n_M < t \wedge \gamma^n_K] =0.
\end{equation*}
Therefore it is enough to show that for any $M>0$,
\begin{equation*}
 \begin{split}
& \varepsilon_n^{-2}\crochet{Z^n}_{\beta^n_M} - 
\frac{1}{6}\int_0^{\beta^n_M}a_u^2(\sigma^Y_u)^2du \to 0, \\
& 
\varepsilon_n^{-1}\crochet{Z^n, Y}_{\beta^n_M}-  \frac{1}{3}\int_0^{\beta^n_M}s_u(\sigma^Y_u)^2du \to 0,
 \end{split}
\end{equation*}
in probability. From the assumptions of Proposition \ref{prop: intuition a s}, we have
\begin{align*}
&
\varepsilon_n^{-1}\sum_{j=0}^{N^n_{\beta^n_M}}\kappa_{\tau^n_j}\mathbb{E}\big[
\Delta_{j, n}^3
 \big| \mathcal{F}_{\tau^n_j}
\big]
+  \int_0^{\beta^n_M}s_u(\sigma^Y_u)^2du \to 0, \\
&
\varepsilon_n^{-2}\sum_{j=0}^{N^n_{\beta^n_M}}\kappa_{\tau^n_j}
\mathbb{E}\big[
\Delta_{j, n}^4
 \big| \mathcal{F}_{\tau^n_j}
\big]
- \int_0^{\beta^n_M}a_u^2(\sigma^Y_u)^2du \to 0,
\end{align*}
in probability.  Moreover, by It$\hat{\text{o}}$'s formula,
\begin{equation*}
\begin{split}
&\frac{1}{3} \varepsilon_n^{-1}\sum_{j=0}^{N^n_{\beta^n_M}}\kappa_{\tau^n_j}\mathbb{E}\big[
\Delta_{j, n}^3
 \big| \mathcal{F}_{\tau^n_j}
\big]
=
\varepsilon_n^{-1}\sum_{j=0}^{N^n_{\beta^n_M}}\kappa_{\tau^n_j}\mathbb{E}\big[
\int_{\tau^n_j}^{\tau^n_{j+1}} (X_u - X^n_{\tau^n_j})d \crochet{X}_u
 \big| \mathcal{F}_{\tau^n_j}
\big], \\
&\frac{1}{6} \varepsilon_n^{-2}\sum_{j=0}^{N^n_{\beta^n_M}}\kappa_{\tau^n_j}\mathbb{E}\big[
\Delta_{j, n}^4
 \big| \mathcal{F}_{\tau^n_j}
\big]
=
\varepsilon_n^{-2}\sum_{j=0}^{N^n_{t\wedge \beta^K}}\kappa_{\tau^n_j}\mathbb{E}\big[
\int_{\tau^n_j}^{\tau^n_{j+1}} (X_u - X^n_{\tau^n_j})^2d \crochet{X} _u
 \big| \mathcal{F}_{\tau^n_j}
\big].
\end{split}
\end{equation*}
Now remark that the following convergences in probability hold:
\begin{equation}\label{lenglart}
 \begin{split}
& \sup_{t\geq 0}\Big| \varepsilon_n^{-1}\sum_{j=0}^{N^n_{\beta^n_M}}\kappa_{\tau^n_j}
\int_{\tau^n_j}^{\tau^n_{j+1}} (X_u - X^n_{\tau^n_j})d \crochet{X}_u
 - \varepsilon_n^{-1}\sum_{j=0}^{N^n_{\beta^n_M}}\kappa_{\tau^n_j}\mathbb{E}\big[
\int_{\tau^n_j}^{\tau^n_{j+1}} (X_u - X^n_{\tau^n_j})d \crochet{X}_u
 \big| \mathcal{F}_{\tau^n_j}
\big] \Big| \to 0,
\\
& \sup_{t\geq 0}\Big| \varepsilon_n^{-2}\sum_{j=0}^{N^n_{\beta^n_M}}\kappa_{\tau^n_j}
\int_{\tau^n_j}^{\tau^n_{j+1}} (X_u - X^n_{\tau^n_j})^2d \crochet{X}_u
 - \varepsilon_n^{-2}\sum_{j=0}^{N^n_{\beta^n_M}}\kappa_{\tau^n_j}\mathbb{E}\big[
\int_{\tau^n_j}^{\tau^n_{j+1}} (X_u - X^n_{\tau^n_j})^2d \crochet{X}_u
 \big| \mathcal{F}_{\tau^n_j}
\big] \Big|\to 0.
 \end{split}
\end{equation}
Indeed, as seen in the proofs of Lemmas~\ref{lem1} and \ref{lem2},
the convergences in probability in \eqref{lenglart} are deduced from the following ones:
\begin{equation}\label{lindeberg}
 \begin{split}
&  \varepsilon_n^{-2}\sum_{j=0}^{N^n_{\beta^n_M}}\kappa_{\tau^n_j}^2\mathbb{E}\big[
\big(
\int_{\tau^n_j}^{\tau^n_{j+1}} (X_u - X^n_{\tau^n_j})d \crochet{X}_u \big)^2
 \big| \mathcal{F}_{\tau^n_j}
\big]  \to 0,\\
&  \varepsilon_n^{-4}\sum_{j=0}^{N^n_{\beta^n_M}}\kappa_{\tau^n_j}^2\mathbb{E}\big[
\big(
\int_{\tau^n_j}^{\tau^n_{j+1}} (X_u - X^n_{\tau^n_j})^2d \crochet{X}_u \big)^2
 \big| \mathcal{F}_{\tau^n_j}
\big]  \to 0.
 \end{split}
\end{equation}
Since $Q_n=\varepsilon_n^{-4}\sup_{t \in [0,T]}|X^n_t-X_t|^4$ is uniformly integrable and
\begin{equation*}
 \sum_{j=0}^{\infty}(\crochet{X}_{\tau^n_{j+1}} - \crochet{X}_{\tau^n_j})^2
\end{equation*}
is bounded and converges to $0$ in probability by Lemma~\ref{lemma: conv diff qv}, we have
\begin{equation*}
 \bbE\Big[
\sum_{j=0}^{N^n_{\beta^n_M}}\kappa_{\tau^n_j}^2
\big(
\int_{\tau^n_j}^{\tau^n_{j+1}} \varepsilon_n^{-k}(X_u - X^n_{\tau^n_j})^kd \crochet{X}_u \big)^2
\Big]
\leq C \bbE \Big[
Q_n^{k/2} \sum_{j=0}^{\infty}(\crochet{X}_{\tau^n_{j+1}} - \crochet{X}_{\tau^n_j})^2
\Big] \to 0
\end{equation*}
for $k=1,2$, which gives \eqref{lindeberg}.\\

\noindent We also have

\begin{equation}\label{edge}
 \begin{split}
& \sup_{t \geq 0}\Big|  \varepsilon_n^{-1}\sum_{j=0}^{N^n_{\beta^n_M}}\kappa_{\tau^n_j}
\int_{\tau^n_j}^{\tau^n_{j+1}} (X_u - X^n_{\tau^n_j})d \crochet{X}_u
-\varepsilon_n^{-1}\sum_{j=0}^{\infty}\kappa_{\tau^n_j}
\int_{\tau^n_j \wedge \beta^n_M}^{\tau^n_{j+1} \wedge \beta^n_M} (X_u - X^n_{\tau^n_j})d \crochet{X}_u  \Big| \to 0,\\
&  \sup_{t \geq 0}\Big|  \varepsilon_n^{-2}\sum_{j=0}^{N^n_{\beta^n_M}}\kappa_{\tau^n_j}
\int_{\tau^n_j}^{\tau^n_{j+1}} (X_u - X^n_{\tau^n_j})^2d \crochet{X}_u
-   \varepsilon_n^{-2}\sum_{j=0}^{\infty}\kappa_{\tau^n_j}
\int_{\tau^n_j \wedge \beta^n_M}^{\tau^n_{j+1}\wedge \beta^n_M} (X_u - X^n_{\tau^n_j})^2d \crochet{X}_u \Big| \to 0,
 \end{split}
\end{equation}
in probability. These two convergences follow using that 
\begin{equation*}
 \sup_{j \geq 0, t \in [0,T]} 
\int_{\tau^n_j  \wedge t}^{\tau^n_{j+1} \wedge t} \varepsilon_n^{-i}|X_u-X^n_u|^i d \crochet{X}_u
\to 0
\end{equation*}
in probability for $i=1,2$, which is deduced from \eqref{loc gamma unif} and the fact that
\begin{equation*}
 \sup_{j \geq 0, t \geq 0} 
\int_{\tau^n_j \wedge \gamma^n_K \wedge t}^{\tau^n_{j+1} \wedge \gamma^n_K \wedge t} \varepsilon_n^{-i}|X_u-X^n_u|^i d \crochet{X}_u
\leq K^i \sup_{j\geq 0}(\crochet{X}_{\tau^n_{j+1}} - \crochet{X}_{\tau^n_j}) \to 0,
\end{equation*}
in probability, by Lemma~\ref{lemma: conv diff qv}.\\

\noindent Finally, remark that the uniform continuity of $\kappa$ and \eqref{unif bound} imply

\begin{equation}\label{disc kappa3}
 \begin{split}
& \sup_{t \geq 0}\Big|  \varepsilon_n^{-1}\sum_{j=0}^{\infty}\kappa_{\tau^n_j}
\int_{\tau^n_j \wedge \beta^n_M}^{\tau^n_{j+1} \wedge \beta^n_M} (X_u - X^n_{\tau^n_j})d \crochet{X}_u
+ \varepsilon_n^{-1}\crochet{Z^n,Y}_{\beta^n_M} \Big| \to 0,\\
&  \sup_{t \geq 0}\Big|  \varepsilon_n^{-2}\sum_{j=0}^{\infty}\kappa_{\tau^n_j}
\int_{\tau^n_j \wedge \beta^n_M}^{\tau^n_{j+1}\wedge \beta^n_M} (X_u - X^n_{\tau^n_j})^2d \crochet{X}_u
-  \varepsilon_n^{-2}\crochet{Z^n}_{\beta^n_M} \Big| \to 0,
 \end{split}
\end{equation}
in probability. Then Proposition~\ref{prop: intuition a s} is eventually obtained from \eqref{lenglart} together with \eqref{edge} and \eqref{disc kappa3}.

\subsection{Proof of Proposition \ref{prop : hitting strats}}

\subsubsection{Proof of the convergence in law to  \eqref{eqn: stable limit}}
We start with the stable convergence in law of the renormalized hedging error.
Such convergence being stable against localization procedures, we can assume without loss of generality that 
$|b^X|$,
$\sigma^X$, $|b^Y|$, $\sigma^Y$, $1/\sigma^Y$, $\ul$, $1/\ul$, $\dl$ and $1/\dl$ are bounded
by a constant $K >0$. 
Then in particular we have $\varepsilon_n^{-1}\sup_{t \in [0,T]}|X^n_t-X_t| \leq K$.\\

\noindent By Lemma~\ref{lemma: condition add}, we have
\eqref{Condition add} for all $T_0 \in [0,T)$.
Furthermore $\varepsilon_n^{-1}\crochet{Z^n,Y}$ and $\varepsilon_n^{-2}\crochet{Z^n}$ are equicontinuous.
Therefore the uniform convergences in probability \eqref{asmp: limit 2} and \eqref{asmp: limit 3} follow from 
the corresponding convergences in probability at each $t \in [0,T)$.\\

\noindent Fix $T_0 \in [0,T)$ and define $\hat{\alpha}^K$ by \eqref{alphahat}.
Then we have \eqref{alphahat conv} and so,
we can assume without loss of generality that
$1/\sigma^X \leq K$ in order to show the
convergences
\eqref{asmp: limit 2} and \eqref{asmp: limit 3} on $[0,T_0]$.
Also, thanks to the Girsanov-Maruyama transformation,
we can assume $b^X=0$.
Define for $\delta >0$ and $t \in [0,T_0]$
\begin{equation*}
 w_t(\delta) = \sup \{ |\ul_u-\ul_v| + |\dl_u-\dl_v| ; 0\leq u \leq t, \ 0\leq v \leq t,\ |u-v|\leq \delta \}.
\end{equation*}
Since $\ul$ and $\dl$ are continuous and bounded, we have
\begin{equation*}
 \bbE[w_{T_0}(\delta)] \to 0
\end{equation*}
as $\delta \to 0$. Let
\begin{equation*}
 T^n_K = \inf\{t > 0; w_t(\varepsilon_n) \geq K \bbE[w_{T_0}(\varepsilon_n)]\}\wedge T_0.
\end{equation*}
Note that
\begin{equation*}
 \sup_{n \in \mathbb{N}}\bbP[T^n_K < T_0]
\leq \sup_{n \in \mathbb{N}}\bbP[
 w_{T_0}(\varepsilon_n) \geq K \bbE[w_{T_0}(\varepsilon_n)]]
\leq \frac{1}{K} \to 0,
\end{equation*}
as $K \to \infty$.
On the set $\{T^n_K<T_0\}$, we can replace $\ul$ and $\dl$ by 
$\ul_{\cdot \wedge T^n_K}$ and
$\dl_{\cdot \wedge T^n_K}$ respectively. This means that we can assume without loss of generality that
$w_{T_0}(\varepsilon_n) \leq K \bbE[w_{T_0}(\varepsilon_n)]$.
Now in order to apply Proposition~\ref{prop: intuition a s},
it remains to show \eqref{asmp: suff}.

\subsubsection*{Part 1: Technical lemma}
We give here a first technical lemma. 
\begin{lem}\label{lem tech}
The sequence $\varepsilon_n^2N^n_{T_0}$ is tight. 
\end{lem}
\begin{proof}
Since
\begin{equation*}
 |X_{\tau^n_{j+1} \wedge T_0}-X_{\tau^n_j \wedge T_0}|^2 \geq \frac{\varepsilon_n^2}{K^2},
\end{equation*}
we have
\begin{equation*}
 \varepsilon_n^2 N^n_{T_0} \leq K^2 \sum_{j=0}^{N^n_{T_0}}
 (X_{\tau^n_{j+1} \wedge T_0}-X_{\tau^n_j \wedge T_0})^2 \to K^2 \crochet{X}_{T_0},
\end{equation*}
in probability by Lemma~\ref{lemma: condition add}.
\end{proof}

\subsubsection*{Part 2: Approximation lemma} 
We give here an important 
 result. Let
$\tilde{\tau}^n_{j+1}$ be the exit time of fixed barriers defined by 
\begin{equation}
\tilde{\tau}^n_{j+1} = \inf\big\{t>\tau^n_j: X_t \notin (X_{\tau^n_j}- \varepsilon_n \dl_{\tau^n_j},X_{\tau^n_j}+\varepsilon_n \ul_{\tau^n_j} )\big\} \wedge T_0.
\end{equation}
We have the following lemma.
\begin{lem}\label{key of lem}
We have
\begin{equation*}
\sum_{j=0}^{N^n_{T_0}}\bbE\big[(\tilde{\tau}^n_{j+1} -\tau^n_{j+1})| \calF_{\tau^n_j}\big]\to 0,
\end{equation*}
in probability.
\end{lem}
\begin{proof}
 Since the sequence $\varepsilon_n^2 N^n_{T_0}$ is tight, it is enough to show that 
\begin{equation*}
\frac{1}{\varepsilon_n^2}\sup_{j\leq N^n_{T_0}}\bbE\big[{\tilde{\tau}^n_{j+1} -\tau^n_{j+1}} \big| \calF_{\tau^n_j}\big]\to 0.
\end{equation*}
We write 
$$\frac{1}{\varepsilon_n^2}\sup_{j\leq N^n_{T_0}}\bbE\big[{\tilde{\tau}^n_{j+1} -\tau^n_{j+1}} \big| \calF_{\tau^n_j}\big]=R_1+R_2,$$
with
\begin{align*}
R_1&=\frac{1}{\varepsilon_n^2}\sup_{j\leq N^n_{T_0}}\bbE\big[({\tilde{\tau}^n_{j+1} -\tau^n_{j+1}})\mathbbm{1}_{\{\tilde{\tau}^n_{j+1}\vee \tau^n_{j+1}\geq\tau^n_j + \varepsilon_n\}}  \big| \calF_{\tau^n_j}\big],\\
R_2&=\frac{1}{\varepsilon_n^2}\sup_{j\leq N^n_{T_0}}\bbE\big[({\tilde{\tau}^n_{j+1} -\tau^n_{j+1}})\mathbbm{1}_{\{\tilde{\tau}^n_{j+1}\vee \tau^n_{j+1}< \tau^n_j + \varepsilon_n\}}  \big| \calF_{\tau^n_j}\big].
\end{align*}
We first treat $R_1$. We have 
$$R_1\leq \frac{T}{\varepsilon_n^2}\sup_{j\leq N^n_{T_0}} \mathbb{P}\big[\tilde \tau^n_{j+1}\vee\tau^n_{j+1}\geq \tau^n_j+ \varepsilon_n\big| \calF_{\tau^n_j}\big].$$
Since $\ul$, $\dl$ and $\sigma^X$ are bounded from below by $1/K$, using the Dambis, Dubins-Schwartz theorem we get that there exists
some $C>0$ such that
\begin{equation*}
\mathbb{P}\big[\tilde\tau^n_{j+1}\vee \tau^n_{j+1}\geq \tau^n_j+\varepsilon_n\big| \calF_{\tau^n_j}\big]\leq 
\mathbb{P}[\rho^n\geq C\varepsilon_n], 
\end{equation*}
with $\rho^n$ the first exit time of $[-\varepsilon_n/K,\varepsilon_n/K]$ by a Brownian motion starting from zero. Using the well-known bound $\bbE[(\rho^n)^k]\leq C\varepsilon_n^{2k}$ for $k\in\mathbb{N}$, Markov's inequality gives the convergence to zero of $R_1$.\\

\noindent We now turn to $R_2$. Recall that $w_{T_0}(\varepsilon_n) \leq K \bbE[w_{T_0}(\varepsilon_n)]=\delta_n \to 0$. Then, we have
$$({\tilde{\tau}^n_{j+1} -\tau^n_{j+1}} ) \mathbbm{1}_ {\{\tilde{\tau}^n_{j+1}\vee \tau^n_{j+1}< \tau^n_j + \varepsilon_n\}}\leq \hat{J}^n_{j+1}-\check{J}^n_{j+1},$$ with
\begin{align*}
\hat{J}^n_{j+1}&=\inf\Big\{t\geq \tau^n_j ; X_{\tau^n_j+t}- X_{\tau^n_j}\notin\big(-\varepsilon_n(\dl_{\tau^n_j}+\delta_n),\varepsilon_n(\ul_{\tau^n_j}+\delta_n)\big)\Big\},\\
\check{J}^n_{j+1}&=\inf\Big\{t\geq \tau^n_j ;X_{\tau^n_j+t}- X_{\tau^n_j}\notin\big(-\varepsilon_n(\dl_{\tau^n_j}-\delta_n),\varepsilon_n(\ul_{\tau^n_j}-\delta_n)\big)\Big\}.
\end{align*}
Using again the Dambis, Dubins-Schwarz theorem and the various boundedness assumptions, we get
$$\bbE\Big[\bbE\big[\hat{J}^n_{j+1} - \check{J}^n_{j+1}\big|\calF_{\check{J}^n_{j+1}}\big]\big|\calF_{\tau^n_j}\Big]\leq C\varepsilon_n^2 \delta_n.$$
Consequently, 
$$\bbE[R_2]\leq C\delta_n,$$
which gives the result.
\end{proof}
\subsubsection*{Part 3: Proof of \eqref{asmp: suff}}
Here we prove \eqref{asmp: suff}, which completes the proof of the convergence in law of $\varepsilon_n^{-1}Z^n_T$ with the help of Proposition~\ref{proposition: stable conv} and Proposition~\ref{prop: intuition a s}.
As already seen, by It$\hat{\text{o}}$'s formula, we have 
\begin{equation*}
\begin{split}
 &  \bbE[\Delta_{j,n}^4|\calF_{\tau^n_j}] = 
6 \bbE\big[
\int_{\tau^n_j}^{\tau^n_j+1} (X_t-X_{\tau^n_j})^2d\crochet{X}_t
\big| \calF_{\tau^n_j}
\big] = A_j, \\
& \bbE[\Delta_{j,n}^3|\calF_{\tau^n_j}] = 
3 \bbE\big[
\int_{\tau^n_j}^{\tau^n_j+1} (X_t-X_{\tau^n_j})d\crochet{X}_t \big| \calF_{\tau^n_j}
\big] = B_j.
\end{split}
\end{equation*}
Therefore, we obtain
\begin{equation*}
\begin{split}
&\varepsilon_n^{-2}\sum_{j=0}^{N^n_t}\kappa_{\tau^n_j}A_j
= \varepsilon_n^{-2}\sum_{j=0}^{N^n_t}\kappa_{\tau^n_j}
\bbE\big[(X_{\tilde{\tau}^n_{j+1}} - X_{\tau^n_j})^4| \calF_{\tau^n_j}\big]
+ R_t
 \\ & 
\varepsilon_n^{-1}\sum_{j=0}^{N^n_t}\kappa_{\tau^n_j}B_j
= \varepsilon_n^{-1}\sum_{j=0}^{N^n_t}\kappa_{\tau^n_j}
\bbE\big[(X_{\tilde{\tau}^n_{j+1}} - X_{\tau^n_j})^3| \calF_{\tau^n_j}\big]
+ R^\prime_t
\end{split}
\end{equation*}
where
\begin{equation*}
\begin{split}
& R_t  =
6 \varepsilon_n^{-2}\sum_{j=0}^{N^n_t}\kappa_{\tau^n_j} 
 \bbE\Big[\int_{\tilde{\tau}^n_{j+1}}^{\tau^n_{j+1}}(X_u-X^n_u)^2(\sigma^X_u)^2du \big| \calF_{\tau^n_j}\Big], \\
& R^\prime_t =3 
\varepsilon_n^{-1}\sum_{j=0}^{N^n_t}\kappa_{\tau^n_j} \bbE\Big[\int_{\tilde{\tau}^n_{j+1}}^{\tau^n_{j+1}}(X_u-X^n_u)(\sigma^X_u)^2du \big| \calF_{\tau^n_j}\Big].
\end{split}
\end{equation*}
Since $\varepsilon_n^{-1}\sup_t|X_t-X^n_t| \leq K$ and $\sigma^X \leq K$, $R$ and $R^\prime$ converge to $0$ uniformly in probability on $[0,T_0]$
by Lemma~\ref{key of lem}.
Using that for $b_1>0$ and $b_2>0$, the probability that a Brownian motion starting from zero hits level $b_1$ before level $-b_2$ is equal to $b_2/(b_2+b_1)$, we get
\begin{equation*}
\varepsilon_n^{-2} \frac{\bbE\big[(X_{\tilde{\tau}^n_{j+1}} - X_{\tau^n_j})^4\big| \calF_{\tau^n_j}\big]}{\bbE\big[(X_{\tilde{\tau}^n_{j+1}} - X_{\tau^n_j})^2\big| \calF_{\tau^n_j}\big]}= a^2_{\tau^n_j},
\ \ 
 \varepsilon_n^{-1} \frac{\bbE\big[(X_{\tilde{\tau}^n_{j+1}} - X_{\tau^n_j})^3\big| \calF_{\tau^n_j}\big]}{\bbE\big[(X_{\tilde{\tau}^n_{j+1}} - X_{\tau^n_j})^2\big| \calF_{\tau^n_j}\big]}= -s_{\tau^n_j},
\end{equation*}
where
\begin{equation*}
a^2=\ul^2 + \dl^2- \ul\dl, \ \ 
s = \dl- \ul. 
\end{equation*}
Then, to complete the proof, it suffices to show that the convergences
\begin{equation*}
 \begin{split}
& \sum_{j=0}^{N^n_\cdot}\kappa_{\tau^n_j}a^2_{\tau^n_j}\bbE\big[(X_{\tilde{\tau}^n_{j+1}} - X_{\tau^n_j})^2| \calF_{\tau^n_j}\big]
\to \int_0^\cdot a_u^2 (\sigma^Y_u)^2du, \\
& \sum_{j=0}^{N^n_\cdot}\kappa_{\tau^n_j}s_{\tau^n_j}\bbE\big[(X_{\tilde{\tau}^n_{j+1}} - X_{\tau^n_j})^2| \calF_{\tau^n_j}\big]
\to \int_0^\cdot s_u (\sigma^Y_u)^2du
 \end{split}
\end{equation*}
hold uniformy in probability on $[0,T_0]$. This follows from Lemma~A.4 in \cite{fukasawa2011discretization} 
together with Lemma~\ref{lemma: conv diff qv}. 

\subsubsection{Proof of \eqref{moment2}} 
Here we prove a moment convergence result. Thus the localization procedure does not apply here. 
We set
\begin{align*}
A_n &= \varepsilon_n^{-1}\int_0^T(X^n_t - X_t)b^Y_t dt,\\
B_n &= \varepsilon_n^{-1}\int_0^T(X^n_t - X_t)\sigma^Y_t dW^Y_t.
\end{align*}
We have 
$$(\varepsilon_n^{-1}{\eps}^n_T)^2 = (A_n+ B_n)^2 \leq 2( A_n^2 + B_n^2). $$
Thus it is enough to prove the uniform integrability of $(A_n^2)$ and $ (B_n^2)$ to obtain the result. For $(A_n^2)$, we have
$$\underset{n}{\text{sup}}(A_n)^2 \leq \big(\int_0^T (\ul_t\vee\dl_t)\abs{b^Y_t}dt\big)^2 \leq \int_0^T(\ul_t\vee\dl_t)^2(\rho_t)^2 (\sigma^Y_t)^2dt.$$
The right hand side of the last inequality being integrable, this gives the result for $(A_n)^2$. We now turn to $(B_n^2)$. The sequence $(B_n^2)$ is non negative integrable and converges in law towards an integrable limit. Thus the uniform integrability is equivalent to the convergence in expectation, see for example \cite{billingsley2009convergence}. Since
\begin{equation*}
\varepsilon_n^{-2}\crochet{Z^n}_T\to \frac{1}{6}\int_0^Ta_t^2(\sigma^Y_t)^2dt
\end{equation*}
and
$$\varepsilon_n^{-2}\crochet{Z^n}_T\leq \int_0^T(\ul_t\vee \dl_t)^2(\sigma^Y_t)^2dt,$$
we readily obtain
\begin{equation*}
\bbE[B_n^2]= \bbE[\varepsilon_n^{-2}\crochet{Z^n}_T]\to \frac{1}{6}\bbE[\int_0^T a_t^2(\sigma^Y_t)^2dt],
\end{equation*} 
which concludes the proof.
\subsection{Proof of Theorem \ref{theo : main result}}
We start with the first part of Theorem \ref{theo : main result}. Let $(m,v)$ be a non dominated couple. Suppose it is a dominated couple with respect to $\mathcal{Z}_T$. This means there exists a process $s_t^*$ such
that the associated expectation, say $m'=\bbE[Z_{T,s^*}]$, is larger than $m$ and the expected error, say $v'=\bbE[(Z_{T,s^*})^2]$, is strictly smaller than $v$. From Lemma \ref{lem : hitting strat univ}, for any $\eta$ we can find an admissible strategy with limiting variable $Z^*_{s^*+\eta,s^*}$. Clearly, we can find $\eta$ small enough, such that 
$\bbE|Z^*_{s^*+\eta,s^*}]=m'$ and  
$$v'\leq \bbE[(Z^*_{s^*+\eta,s^*})^2]<v.$$
Consequently $(m,v)$ is a dominated couple, which is absurd. Conversely, any point which is non dominated with respect to $\mathcal{Z}_T$ is non dominated since $a_t^2\geq s_t^2$.\\

\noindent For the second part, it remains to show that the proposed discretization rules indeed lead to nearly efficient couples. The fact that they are admissible is clear from Proposition \ref{prop : hitting strats}. Recall now that for the suggested rule
$$a_t^2=(s_t^*)^2+\frac{6\delta}{(\sigma^Y_t)^2}.$$
This equality
gives that the limiting variable $Z^*_{a,s}$ associated to this discretization rule satisfies
$$\bbE[Z^*_{a,s}]= \frac{1}{3}\bbE\big[\int_0^Ts_t^*dY_t\big]$$
and
\begin{align*}
\bbE[(Z^*_{a,s})^2]&=\frac{1}{9}\bbE\big[\big(\int_0^Ts_t^*dY_t\big)^2\big]+\frac{1}{6}\bbE \big[\int_0^T\big(a_t^2-\frac{2}{3}(s_t^*)^2\big)(\sigma^Y_t)^2dt\big]\\
&=\frac{1}{9}\bbE\big[\big(\int_0^Ts_t^*dY_t\big)^2\big]+\frac{1}{18}\bbE \big[\int_0^T\big((s_t^*)^2(\sigma^Y_t)^2\big)dt\big]+\delta T.
\end{align*}
The couple
$$\Big(\frac{1}{3}\bbE\big[\int_0^Ts_t^*dY_t\big],\frac{1}{9}\bbE\big[\big(\int_0^Ts_t^*dY_t\big)^2\big]+\frac{1}{18}\bbE \big[\int_0^T\big(\frac{1}{2}(s_t^*)^2(\sigma^Y_t)^2\big)dt\big]\Big)$$
being non dominated, we obtain the result.

%

\section{Linear-quadratic optimal control}\label{append: LQ}
We give here a summary of useful formulas from \cite{zhou2000continuous}. Consider a controlled system governed by the following linear SDE:
\begin{equation}\label{LQ: dynm}
\begin{cases}
dX_t = (A_tX_t + B_tu_t + f_t)dt + \sum_{j=1}^mD^j_tu_tdW^j_t,\\
X_0 = x \in \bbR^n,
\end{cases}
\end{equation}
where $x$ is the initial state and $W =(W^1, \cdots, W^m)$ is a $m$-dimensional Brownian motion on a given filtered probability space $(\Omega, \calF, \bbP, \paren{\calF_t}_{t\geq 0})$ and $u\in L^2_{\calF}([0, T], \bbR^m)$ is a control. For each control $u$, the associated cost is 
\begin{equation}\label{LQ: cost}
J(u) = \E{\int_0^T\frac{1}{2}\paren{X_t'Q_tX_t + u'_tR_tu_t}dt + \frac{1}{2}X_T'HX_T}.
\end{equation}
We suppose that all the parameters are deterministic and continuous on $[0, T]$ and $H$ belongs to $S^n_+$ the set of $n\times n$ symmetric positive matrices. We introduce the following matrix Riccati equation
 \begin{equation}\label{eqn: riccati}
 \begin{cases}
\dot{P}_t = -P_tA_t -A_t'P_t -Q_t +P_tB_tK_t^{-1}B_t'P_t,\\
P_T = H,\\
K_t = R_t + \sum_{j=1}^{m}D^{j'}_tP_tD^j_t > 0, \quad \forall t\in [0, T],
 \end{cases}
 \end{equation}
along with an equation
\begin{equation}\label{eqn: g}
\begin{cases}
\dot{g}_t = -A_t'g_t + P_tB_tK_t^{-1}B_t'g_t-P_tf_t,\\
g_T = 0.
\end{cases}
\end{equation} 
Then following result is given in \cite{zhou2000continuous}.
\begin{theo}\label{theoxyz}
If \eqref{eqn: riccati} and \eqref{eqn: g} admit solutions $P\in C([0,
T], S^n_+)$ and $g\in C([0, T], \bbR^n)$ respectively, then the
stochastic linear-quadratic control problem \eqref{LQ: dynm}-\eqref{LQ: cost} has an optimal feedback control
\begin{equation*}
u^*(t, x)=-K_t^{-1}B'_t(P_tX_t+g_t). 
\end{equation*}
Moreover, the optimal cost value is 
\begin{equation*}
J^* = \frac{1}{2}\int_{0}^{T}\paren{2f_t'g_t - g_tB_tK^{-1}_tB_t'g_t}dt + \frac{1}{2}x'P_0x + xg_0.
\end{equation*}
\end{theo}
\end{appendices}

\subsection*{Acknowledgements}

We thank Philippe Amzelek and Joe Bonnaud from BNP-Paribas for inspiring discussions.

\bibliographystyle{abbrv}
\bibliography{biblio_cfrt}
\end{document}